\documentclass[a4paper, reqno, 10pt]{amsart}

\usepackage{verbatim}
\usepackage[dvips]{color}
\usepackage[usenames,dvipsnames]{xcolor}
\usepackage{amssymb}
\usepackage[active]{srcltx}
\usepackage[colorlinks,pdfpagelabels,pdfstartview = FitH,bookmarksopen = true,bookmarksnumbered = true,linkcolor = NavyBlue,plainpages = false,hypertexnames = false,citecolor = OliveGreen] {hyperref}
\usepackage{mathtools}
\usepackage{empheq}
\usepackage{a4wide}
\usepackage{aligned-overset}

\makeatletter
\@namedef{subjclassname@2020}{\textup{2020} Mathematics Subject Classification}
\makeatother

\allowdisplaybreaks

\title[Singular elliptic measure data problems with double phase]{Gradient estimates for singular elliptic measure data problems with double phase}

\author[Song]{Kyeong Song}
\address{School of Mathematics,
Korea Institute for Advanced Study, Seoul 02455, Republic of Korea}
\email{kyeongsong@kias.re.kr}

\author[Youn]{Yeonghun Youn}
\address{Department of Mathematics Education,
Incheon National University, Incheon 21999, Republic of Korea}
\email{yeonghunyoun@inu.ac.kr}

\author[Zatorska-Goldstein]{Anna Zatorska-Goldstein}
\address{University of Warsaw, Institute of Applied Mathematics, and Mechanics, ul. Banacha 2, 02-097 Warsaw, Poland}
\email{azator@mimuw.edu.pl}



\newtheorem{theorem}{Theorem}[section]

\newtheorem{lemma}[theorem]{Lemma}

\newtheorem{corollary}[theorem]{Corollary}
\theoremstyle{definition}
\newtheorem{definition}[theorem]{Definition}
\newtheorem{remark}[theorem]{Remark}

\numberwithin{equation}{section}

\def\eqn#1$$#2$${\begin{equation}\label#1#2\end{equation}}
\def\charfn_#1{{\raise1.2pt\hbox{$\chi_{\kern-1pt\lower3pt\hbox{{$\scriptstyle#1$}}}$}}}

\makeatletter
\newcommand{\pushright}[1]{\ifmeasuring@#1\else\omit\hfill$\displaystyle#1$\fi\ignorespaces}
\newcommand{\pushleft}[1]{\ifmeasuring@#1\else\omit$\displaystyle#1$\hfill\fi\ignorespaces}
\makeatother

\DeclareMathOperator*{\osc}{osc}

\newcommand{\divo}{\textnormal{div}}

\DeclareMathOperator*{\data}{{\mathtt{data}}}

\newcommand{\ern}{\mathbb{R}^n}

\def\loc{{\operatorname{loc}}}

\newcommand{\dx}{\,dx}

\newcommand{\means}[1]{-\hskip-1.00em\int_{#1}}

\def\mean#1{\mathchoice%
          {\mathop{\kern 0.2em\vrule width 0.6em height 0.69678ex depth -0.58065ex
                  \kern -0.8em \intop}\nolimits_{\kern -0.4em#1}}%
          {\mathop{\kern 0.1em\vrule width 0.5em height 0.69678ex depth -0.60387ex
                  \kern -0.6em \intop}\nolimits_{#1}}%
          {\mathop{\kern 0.1em\vrule width 0.5em height 0.69678ex
              depth -0.60387ex
                  \kern -0.6em \intop}\nolimits_{#1}}%
          {\mathop{\kern 0.1em\vrule width 0.5em height 0.69678ex depth -0.60387ex
                  \kern -0.6em \intop}\nolimits_{#1}}}

\newcommand{\vertiii}[1]{{\left\vert\kern-0.25ex\left\vert\kern-0.25ex\left\vert #1 
			\right\vert\kern-0.25ex\right\vert\kern-0.25ex\right\vert}}
			
\def\avenorm#1{\mathchoice%
          {\mathop{\kern 0.2em\vrule width 0.6em height 0.69678ex depth -0.58065ex
                  \kern -0.545em \|{#1}\|}}%
          {\mathop{\kern 0.1em\vrule width 0.5em height 0.69678ex depth -0.60387ex
                  \kern -0.495em \|{#1}\|}}%
          {\mathop{\kern 0.1em\vrule width 0.5em height 0.69678ex depth -0.60387ex
                  \kern -0.495em \|{#1}\|}}%
          {\mathop{\kern 0.1em\vrule width 0.5em height 0.69678ex depth -0.60387ex
                  \kern -0.495em \|{#1}\|}}}

\newtoks\by
\newtoks\paper
\newtoks\book
\newtoks\jour
\newtoks\yr
\newtoks\pages
\newtoks\vol
\newtoks\publ

\def\ota{{\hbox{\bf ???}}}
\def\cLear{\by=\ota\paper=\ota\book=\ota\jour=\ota\yr=\ota
\pages=\ota\vol=\ota\publ=\ota}
\def\endpaper{\the\by, \textit{\the\paper},
{\the\jour} \textbf{\the\vol} (\the\yr), \the\pages.\cLear}
\def\endbook{\the\by, \textit{\the\book},
\the\publ, \the\yr.\cLear}
\def\endpap{\the\by, \textit{\the\paper}, \the\jour.\cLear}
\def\endproc{\the\by, \textit{\the\paper}, \the\book, \the\publ,
\the\yr, \the\pages.\cLear}

\begin{document}

\begin{abstract}
We consider elliptic measure data problems of the type
\[ -\mathrm{div}\,(|Du|^{p-2}Du+a(x)|Du|^{q-2}Du) = \mu \]
in a bounded domain in $\mathbb{R}^n$, where $p<q$ and $a(\cdot) \ge 0$. We prove local Calder\'on--Zygmund estimates in the singular case $2-1/n < p < 2$, under natural assumptions on $p$, $q$ and $a(\cdot)$. 
\end{abstract}

\maketitle


\section{Introduction}
In this paper, we study gradient regularity for nonlinear elliptic equations modeled on
\begin{equation}\label{model}
-\mathrm{div}(|Du|^{p-2}Du + a(x)|Du|^{q-2}Du) = \mu \quad \text{in}\;\; \Omega.
\end{equation}
Here, $\Omega \subset \ern$ ($n \ge 2$) is a bounded domain and $\mu$ is a signed Borel measure on $\Omega$ with finite total mass $|\mu|(\Omega)<\infty$; in the following, we will consider $\mu$ as a measure on $\mathbb{R}^n$ by letting $|\mu|(\mathbb{R}^n \setminus \Omega) = 0$. 
We initially assume that $p < q $ and $0 \le a(\cdot) \in L^{\infty}(\Omega)$. 
The leading operator in \eqref{model} naturally appears in the Euler--Lagrange equation of the double phase functional 
\begin{equation}\label{dp.ftnal}
w \mapsto \int_{\Omega} \left(\frac{|Dw|^p}{p} + a(x)\frac{|Dw|^q}{q}\right) \dx.
\end{equation}
This type of functional was first introduced by Zhikov \cite{Zhi86} to describe the behavior of strongly anisotropic materials in the context of homogenization. 
It also provides important examples in the study of the Lavrentiev phenomenon \cite{ Zhi95}. 
The main feature of the functional \eqref{dp.ftnal} is that its energy density exhibits a drastic change of growth/ellipticity depending on the point $x \in \Omega$. 
Namely, the energy density of \eqref{dp.ftnal} has $p$-growth in the gradient on the set $\{a(x)=0\}$, while it exhibits $q$-growth on the set $\{a(x) > 0\}$. 
Indeed, the functional \eqref{dp.ftnal} is one of the prominent examples of problems with nonstandard growth and nonuniform ellipticity, see \cite{MR21} for a well-written summary. 

Regularity for double phase problems was first systematically studied in the pioneering papers \cite{BCM15,BCM18,CM15a, CM15b} by Baroni, Colombo and Mingione, who established maximal regularity (i.e., gradient H\"older regularity) results for local minimizers of functionals modeled on \eqref{dp.ftnal}. 
They controlled the phase transition by using a balance condition between the closeness of $p$ and $q$ and the regularity of $a(\cdot)$, see \eqref{rate.weak.sol}, which is shown to be sharp \cite{ELM}. 
Since then, the regularity theory for double phase problems has been actively investigated in several directions, see for instance \cite{BL22,BO,CD,CM16JFA,DM19,DM20,OSS} and references therein.  
In particular, the papers \cite{BL22, BO,CM16JFA,DM19} are concerned with Calder\'on--Zygmund type gradient estimates for weak solutions to inhomogeneous double phase equations. However, most of the existing results  are confined to the case when the data are sufficiently regular to ensure the existence of weak solutions in the natural energy space, and regularity issues  for double phase problems with measure data have remained largely open.

In this paper, we are interested in gradient integrability results for \eqref{model} in the realm of nonlinear Calder\'on--Zygmund theory. 
Let us first mention some known results in this direction. Consider, for instance, elliptic measure data equations involving the standard $p$-Laplacian:
\begin{equation}\label{measure.p}
-\divo\, (|Du|^{p-2}Du) = \mu.
\end{equation}
Mingione \cite{Min10} employed the 1-fractional maximal function of $\mu$, defined by
\begin{equation}\label{def.M1}
\mathbf{M}_1(\mu)(x) \coloneqq \sup_{r>0} \frac{|\mu|(B_r(x))}{r^{n-1}} \qquad \text{for}\;\; x \in \ern,
\end{equation}
to obtain local gradient estimates for \eqref{measure.p} in the scale of Lorentz--Morrey spaces, when  $p \ge 2$. Phuc \cite{Phuc} further developed this approach to get global weighted gradient estimates, with applications to related Riccati type equations, for the case $p > 2-1/n$.

Subsequently, similar kinds of gradient estimates have been established for elliptic measure data problems with various kinds of nonstandard growth, such as Orlicz growth \cite{BCY21a}, mild phase transition \cite{BCY21b}, and variable exponent growth \cite{BOP17}. 
However, the techniques in \cite{BCY21a, BCY21b, BOP17} cannot be directly applied to the case of double phase operator.
Its nonuniform ellipticity, which is stronger than those of the operators considered in \cite{BCY21b,BOP17}, causes substantial difficulties in establishing regularity estimates below the natural energy space. 
For instance, as typical in nonuniformly elliptic problems, 
several estimates obtained in \cite{CM15b,CM16JFA} involve constants depending on the $L^p$-norm of the gradient, which is no longer available in the setting of measure data problems. 

To the best of our knowledge, gradient regularity for double phase problems with measure data was first investigated in the recent paper \cite{SY} by the first and second authors of this paper. 
Specifically, in \cite{SY}, local Calder\'on--Zygmund type estimates for SOLA (Solutions Obtained as Limits of Approximations) were obtained in the degenerate case $p \ge 2$. 
To avoid the aforementioned $L^p$-norm dependency, the approach of \cite{SY}  employed new reverse H\"older estimates for homogeneous double phase equations, where the involved constants depend only on the $L^{\kappa}$-norm of the gradient for some $\kappa<p$ rather than the $L^{p}$-norm, along with comparison estimates. The paper \cite{SY} also introduced a new structural assumption for the validity of such results, see \eqref{rate.sola}. 
However, some techniques in \cite{SY} are confined to the degenerate case, and the singular case requires additional new ideas; see Section \ref{sec:tech} for details.

The aim of this paper is to establish Calder\'on--Zygmund type estimates for elliptic measure data problems of the type \eqref{model} in the singular case $2-1/n < p < 2$, which provide the complete singular counterpart to the results in \cite{SY} within the framework of SOLA. 

\subsection{Assumptions and main result}
Let us introduce the main assumptions of this paper. We consider the following Dirichlet problem:
\begin{equation}\label{main.eq}
\left\{
\begin{aligned}
-\divo\, A(x,Du) &= \mu & \text{in }& \Omega, \\
u &= 0 & \text{on }& \partial\Omega.
\end{aligned}
\right.
\end{equation}
The vector field $A: \Omega \times \ern \to \ern$ is assumed to be continuously differentiable with respect to the second variable $z \in \ern \setminus \{0\}$, with $\partial A(\cdot) \equiv \partial_z A(\cdot)$ being Carath\'eodory regular. It is further assumed to satisfy the following growth, ellipticity, and continuity conditions:
\begin{equation}\label{growth}
\begin{cases}
|A(x, z)| + |\partial A(x, z)||z| \le L(|z|^{p-1} + a(x)|z|^{q-1}), \\
\nu(|z|^{p-2} + a(x)|z|^{q-2})|\zeta|^2 \le \partial A(x, z)\zeta\cdot\zeta, \\
|A(x_1, z) - A(x_2, z)| \le L|a(x_1) - a(x_2)||z|^{q-1}
\end{cases}
\end{equation}
for any $x, x_1, x_2 \in \Omega$, $z \in \ern \setminus \{0\}$, and $\zeta \in \ern$, where $0<\nu \leq L<\infty$ are fixed constants and $a:\Omega \to [0,\infty)$ is a modulating coefficient satisfying
\begin{equation}\label{a.holder}
0 \le a(x) \le \|a\|_{L^\infty} \;\; \text{for a.e.}\;\; x \in \Omega, \qquad \text{and} \qquad a \in C^{0,\alpha}(\Omega)\;\; \text{for some}\;\; \alpha \in (0, 1].
\end{equation}
Throughout the paper, unless otherwise specified, we consider the singular case
\begin{equation}\label{p.bound}
2-1/n < p < 2 \qquad \text{and} \qquad p < q < \infty.
\end{equation}
To control the rate of nonuniform ellipticity of \eqref{main.eq} in the setting of measure data problems, we assume that
\begin{equation}\label{rate.sola}
\frac{q-1}{p-1} < 1 + \frac{\alpha}{n-1}.
\end{equation}
As discussed in \cite[Remark 1.1]{SY}, condition \eqref{rate.sola} can be seen as the natural measure data analog of the well-known bound
\begin{equation}\label{rate.weak.sol}
\frac{q}{p} \le 1 + \frac{\alpha}{n},
\end{equation}
which is a central assumption used in \cite{BCM18, CM15b, CM16JFA, DM19} for the validity of regularity of minimizers or weak solutions. 
Heuristically, when $a(x)$ is close to $0$, assumption \eqref{rate.weak.sol} allows one to appropriately control the double phase functional \eqref{dp.ftnal} in terms of the $W^{1,p}$-energy. 
In the same spirit, our assumption \eqref{rate.sola} precisely ensures that several regularity estimates for \eqref{main.eq} can be reduced to those for measure data problems with $p$-growth when $a(x)$ is close to $0$.

In view of \eqref{model} and \eqref{dp.ftnal}, we denote
\begin{equation}\label{def.H}
H(x,t) \coloneqq t^{p} + a(x) t^{q} \qquad \text{and} \qquad h(x, t) \coloneqq t^{p-1} + a(x)t^{q-1}
\end{equation}
for $x \in \Omega$ and $t \ge 0$. 
Basic properties of the functions $H$ and $h$, along with related function spaces, will be discussed in the next section.  

As mentioned above, we consider the notion of SOLA introduced in \cite{BG89}. 
Here we state the definition of SOLA to \eqref{main.eq}.

\begin{definition}\label{def:sola}
A function $u \in W^{1,1}_0(\Omega)$ is a SOLA to the problem \eqref{main.eq}, under assumptions \eqref{growth} with
\begin{equation}\label{growth.full}
2 - 1/n < p \leq n \qquad \text{and} \qquad p < q < \infty,
\end{equation}
if it is a distributional solution to $\eqref{main.eq}_1$, i.e., $A(x, Du) \in L^{1}(\Omega)$ and 
\[
\int_{\Omega} A(x, Du) \cdot D \varphi \dx = \int_{\Omega} \varphi \, d \mu
\]
holds for any $\varphi \in C^{\infty}_{0}(\Omega)$.
Moreover, it has to satisfy the following approximation property: there exists a sequence of weak solutions $\{u_j\} \subset W^{1,H}_{0}(\Omega)$ to
\begin{equation}\label{approx.prob}
\left\{
\begin{aligned}
-\text{div}\,A(x,Du_{j}) & = \mu_{j}& \text{in }& \Omega, \\
u_{j} & = 0& \text{on }& \partial\Omega
\end{aligned}
\right.
\end{equation}
such that $u_j \to u$ in $W^{1,1}_{0}(\Omega)$. Here, the sequence $\{\mu_j\} \subset L^{\infty}(\Omega)$ converges to $\mu$ weakly$^*$ in the sense of measures and satisfies
\begin{equation}\label{muj.conv}
\limsup_{j\to\infty}|\mu_j|(B) \le |\mu|(\overline{B}) \quad \text{for any ball}\;\; B\subset \mathbb{R}^n.
\end{equation}
\end{definition}

It is proved in \cite[Proposition 2.3]{SY} that, under assumptions \eqref{growth}, \eqref{a.holder}, \eqref{rate.sola} and \eqref{growth.full}, there exists a SOLA $u$ to \eqref{main.eq} such that $u \in W^{1,s}_{0}(\Omega)$ for any $s$ satisfying 
\begin{equation}\label{exp.range}
1 \le s < \frac{n(p-1)}{n-1}.
\end{equation}
This is shown by constructing a sequence of approximating solutions $\{u_j\}$ to \eqref{approx.prob} which satisfies $u_j \to u$ strongly in $W^{1,s}_0(\Omega)$ for any $s$ satisfying \eqref{exp.range}. 
We remark that its proof particularly employs the observation that
\begin{equation}\label{q-1.admissible}
\eqref{rate.sola},\;\; \alpha \le 1 \quad \Longrightarrow \quad \frac{q-1}{p-1} < 1+\frac{1}{n-1} \quad \Longleftrightarrow \quad q-1 < \frac{n(p-1)}{n-1}.
\end{equation}

We now state our main theorem.
To avoid the burden of writing the dependencies of various universal constants, we denote the set of main structural parameters by
\[ \data \coloneqq (n,p,q,\nu,L,\|a\|_{L^{\infty}},\alpha,[a]_{\alpha}). \]
Also, the exponents $\kappa \equiv \kappa(n,p,q,\alpha)$ and $\gamma_0 \equiv \gamma_0(n,p,q,\alpha)$ appearing in the following theorem are given in \eqref{def.kappa} and \eqref{def.gamma0}, respectively.

\begin{theorem}\label{main.thm}
Let $u$ be a SOLA to \eqref{main.eq} under assumptions \eqref{growth}--\eqref{rate.sola}. Then we have the implication
\begin{equation}\label{cz.implication}
\mathbf{M}_1(\mu) \in L^\gamma_{\loc}(\Omega) \implies h(\cdot, |Du|) \in L^\gamma_{\loc}(\Omega) \quad \text{for any}\;\; \gamma \in (1, \infty).
\end{equation}
Moreover, for any $\gamma \in (1, \infty)$, there exist a radius $R_{0} \equiv R_{0}(\data,|\mu|(\Omega), \|Du\|_{L^{\kappa}(\Omega)}, \|Du\|_{L^{\gamma_0}(\Omega)},\gamma) \in (0,1)$ and constants $c\equiv c(\data, \|Du\|_{L^{\kappa}(\Omega)}) \ge 1$ and $c_{\gamma} \equiv c_{\gamma}(\data, \|Du\|_{L^{\kappa}(\Omega)} ,\gamma) \ge 1$ such that
\begin{equation}\label{main.est}
\left( \means{B_R} [h(x, |Du|)]^\gamma \dx \right)^{1/\gamma} \le c \means{B_{2R}} h(x, |Du|) \dx + c_\gamma \left( \means{B_{2R}} [\mathbf{M}_1(\mu)]^\gamma \dx \right)^{1/\gamma}
\end{equation}
holds whenever $B_{2R} \Subset \Omega$ is a ball with $2R \in (0, R_0)$. 
\end{theorem}

\begin{remark}
A careful inspection of the proofs shows that when $q \ge 2$, the radius $R_0$ can be taken independent of $\|Du\|_{L^{\gamma_0}(\Omega)}$.
See the proof of Lemma \ref{lem.1st.h} below.
\end{remark}

\subsection{Technical novelties}\label{sec:tech}
The approach of the present paper is based on that of \cite{SY}.
However, extending the methods of \cite{SY} to the singular case is not a trivial task, since this case causes several difficulties which did not appear in the degenerate case.
These difficulties stem from the fact that the operator exhibits different monotonicity properties compared to the degenerate case. Such issues are particularly intricate in the setting of measure data problems.

One of the delicate issues is concerned with comparison estimates between \eqref{main.eq} and a Dirichlet problem involving
\begin{equation}\label{intro.homo}
-\mathrm{div}\,A(x,Dw) = 0 \quad \text{in}\;\; B \Subset \Omega,
\end{equation}
which play a crucial role in establishing various regularity results for \eqref{main.eq}. In this paper, we obtain new comparison estimates below the natural energy space that extend those available for the $p$-Laplacian \eqref{measure.p} in a precise way.
In contrast to those concerned with the degenerate case \cite[Step 2 in the proof of Theorem 1.2]{SY}, the comparison estimates obtained in this paper involve additional terms depending on $Du$.
This is a common phenomenon which already appears in the case of \eqref{measure.p}, see for instance \cite{DM10JFA}. 
Indeed, it is one of the delicate points of the proof to deal with such terms in a way suitable in the double phase setting. 
This forces us to perform a careful analysis with the help of \eqref{rate.sola}, treating the $p$-phase (the case when $a(x)$ is close to $0$) and the $(p,q)$-phase (the case when $a(x)$ is far from $0$) separately.

Moreover, in order to obtain estimates for $h(x,|Du-Dw|)$, we additionally show a reverse H\"older type estimate for \eqref{main.eq} (Lemma \ref{lem.rev.u}) by modifying the idea used in \cite[Proposition~4.1]{KM12JFA}.
The proof goes by combining a preliminary comparison estimate (Lemma \ref{lem.comp.pre}) and a reverse H\"older estimate for \eqref{intro.homo} (Corollary \ref{revhol.cor}).
For the latter, we also extend the regularity estimates for \eqref{intro.homo}, obtained in \cite[Section 3]{SY}, to the singular case.
In particular, we present a more streamlined proof which covers both degenerate and singular cases simultaneously.
Then, using all the aforementioned ingredients, we prove our main result in Theorem \ref{main.thm} via the maximal function free technique introduced in \cite{AM} and revisited in \cite{BCY21a,BCY21b,SY}.

We expect that our approach and results can be extended in several directions, as mentioned in \cite[Remark~1.3]{SY}.
We also remark that the lower bound $p>2-1/n$ is essential in dealing with SOLA (recall \eqref{exp.range}), and one needs different notions of solutions in order to consider lower values of $p$, see for instance \cite{BBGGPV95,Ch18,CiMa17NA,DMOP99} and references therein.
Gradient estimates for $p$-Laplacian type equations with measure data in the case $1 < p \le 2-1/n$ can be found in \cite{NP19,NP23ARMA}.

The remaining part of this paper is organized as follows.
In the next section, we introduce basic notations and function spaces.
In Section \ref{sec3}, we obtain reverse H\"older type estimates and gradient integrability results for homogeneous problems linked to \eqref{main.eq}.
In Section \ref{sec.comp}, we establish comparison estimates for \eqref{main.eq} and related homogeneous problems.
Finally, in Section \ref{sec4}, we prove Theorem \ref{main.thm}.

\section{Preliminaries} 
\subsection{Notation}
We denote by $c$ a general constant greater than or equal to one, 
whose value may vary from line to line. Specific dependencies of constants are denoted by using parentheses.

For any $p >1$, we denote its H\"older conjugate by $p' \coloneqq p/(p-1)$. Moreover,
\begin{equation*}
p^{*} \coloneqq 
\left\{
\begin{aligned}
&\frac{np}{n-p} & \text{if}\;\; & p < n, \\
&\text{any number in }(p,\infty) & \text{if}\;\; & p \ge n
\end{aligned}
\right.
\end{equation*}
and 
\[ p_{*} \coloneqq \max\left\{\frac{np}{n+p},1\right\} \]
denote the Sobolev conjugate and the inverse Sobolev conjugate of $p$, respectively.

As usual, we denote by $B_{r}(x_0) \coloneqq \{x\in\mathbb{R}^{n}:|x-x_0|<r\}$ the $n$-dimensional open ball with center $x_0 \in \mathbb{R}^n$ and radius $r>0$. If there is no confusion, we omit the center and simply denote $B_{r}\equiv B_{r}(x_0)$. 
Also, given a ball $B$, we denote by $\gamma B$ the concentric ball with radius magnified by a factor $\gamma>0$. 
Unless otherwise stated, different balls in the same context are concentric. 

The $n$-dimensional Lebesgue measure of a measurable set $U \subset \mathbb{R}^n$ is denoted by $|U|$. For an integrable map $f:U \to \mathbb{R}^{k}$, with $k \in \mathbb{N}$ and $0<|U|<\infty$, we denote the integral average of $f$ over $U$ by
\begin{equation*}
(f)_{U} \coloneqq \mean{U}f\,dx  \coloneqq \frac{1}{|U|} \int_{U}f\,dx .
\end{equation*}

\subsection{Function spaces}
We recall the function $H(\cdot)$ defined in \eqref{def.H}. It is a generalized Young function (see \cite{HH19book}) that satisfies the $\Delta_2$ and $\nabla_2$ conditions: $H(x,2t) \le 2^q H(x,t)$ and $H(x,t) \le 2^{-p/(p-1)}H(x,2^{1/(p-1)}t)$ for any $x \in \Omega$ and $t \ge 0$. Accordingly, for any open set $U \subseteq \Omega$,
we consider the Musielak--Orlicz space $L^{H}(U)$, which is defined as the set of all measurable function $f:U \to \mathbb{R}$ such that
\[ \int_{U}H(x,|f|)\,dx < \infty. \]
Due to the $\Delta_2$ and $\nabla_2$ conditions, we have that $L^{H}(U)$ is a separable, reflexive Banach space endowed with the Luxemburg norm
\[ \|f\|_{L^{H}(U)} \coloneqq \inf\left\{ \lambda>0: \int_{U}H(|f|/\lambda)\,dx < \infty \right\}. \] 
Moreover, the Musielak--Orlicz--Sobolev space $W^{1,H}(U)$ is defined as the set of all functions $f \in L^{H}(U) \cap W^{1,1}(U)$ satisfying
\[ \int_{U}H(x,|Df|)\,dx < \infty, \]
which is also a separable reflexive Banach space endowed with the norm
\[ \|f\|_{W^{1,H}(U)} \coloneqq \|f\|_{L^H(U)} + \|Df\|_{L^H(U)}. \]
We further define $W^{1,H}_{0}(U)$ as the closure of $C^{\infty}_{0}(U)$ in $W^{1,H}(U)$. 
We refer to \cite{HH19book} for more on generalized Young functions and Musielak--Orlicz spaces. 

\subsection{An auxiliary vector field}
Observe that the ellipticity assumption $\eqref{growth}_{2}$ on $A(\cdot)$ implies the following monotonicity property:
\begin{equation*}
(|z_{1}|+|z_{2}|)^{p-2}|z_{1}-z_{2}|^{2} + a(x)(|z_{1}|+|z_{2}|)^{q-2}|z_{1}-z_{2}|^{2} \le c(A(x,z_{1})-A(x,z_{2}))\cdot(z_{1}-z_{2})
\end{equation*}
for any $x \in \Omega$ and $z_{1},z_{2} \in \mathbb{R}^{n}$, where $c\equiv c(n,p,q,L,\nu)$.

We accordingly introduce, for $s \in \{p,q\}$, the auxiliary vector field $V_{s} : \mathbb{R}^n \to \mathbb{R}^n$ defined by
\[ V_{s}(z) \coloneqq |z|^{(s-2)/2}z, \qquad z \in \mathbb{R}^n. \]
A basic property of $V_s$ is that
\begin{equation}\label{VV}
|V_{s}(z_1) - V_{s}(z_2)|^{2} \approx (|z_1| + |z_2|)^{s-2}|z_1 - z_2|^{2}
\end{equation}
holds for any $z_1 , z_2 \in \mathbb{R}^n$, with the implicit constants depending only on $n$ and $s$. 
In particular, we have
\begin{equation}\label{Vprop}
\begin{aligned}
& s\ge 2 \;\; \Longrightarrow \;\; |z_1 - z_2 |^{s} \le c|V_{s}(z_1)-V_{s}(z_2)|^{2}, \\
& s<2 \;\; \Longrightarrow \;\; |z_1 - z_2|^{s} \le c\varepsilon^{(s-2)/s}|V_{s}(z_1) - V_{s}(z_2)|^{2} + \varepsilon|z_1|^{p} \quad \forall\; \varepsilon \in (0,1),
\end{aligned}
\end{equation}
where $c\equiv c(n,s)$. 
In light of \eqref{VV}, the above monotonicity property of $A(\cdot)$ is equivalent to
\begin{equation}\label{mono}
|V_{p}(z_1) - V_{p}(z_2)|^2 + a(x)|V_{q}(z_1) - V_{q}(z_2)|^2 \le c(A(x,z_1)-A(x,z_2))\cdot (z_1 - z_2 ).
\end{equation}

We end this section with a well-known iteration lemma, see for instance \cite[Lemma 6.1]{Giu}.

\begin{lemma}\label{tech.lemma}
Let $Z:[R,2R] \rightarrow [0,\infty)$ be a bounded function that satisfies
\begin{equation*}
Z(r_{1}) \le \delta_{0}Z(r_{2}) + \frac{C_{1}}{(r_{2}-r_{1})^{\ell}} + C_{2}
\end{equation*}
for any $r_{1},r_{2}$ with $R \le r_{1} < r_{2} \le 2R$, where $\delta_{0} \in (0,1)$ and $C_{1},C_{2},\ell > 0$ are given constants. Then there exists a constant $c \equiv c(\delta_{0},\ell)$ such that
\begin{equation*}
Z(R) \le c\left[\frac{C_{1}}{R^{\ell}} + C_{2}\right].
\end{equation*}
\end{lemma}

\section{Regularity for reference problems}\label{sec3}
In this section, we investigate several regularity results for homogeneous equations linked to \eqref{main.eq}. Throughout this section, we consider the range \eqref{growth.full}, which covers both the singular and degenerate cases.

\subsection{Gradient integrability for homogeneous equations}
Here we extend the gradient regularity results for homogeneous double phase problems, established in \cite[Section 3]{SY}, to the singular case \eqref{p.bound}. In particular, as mentioned in the introduction, we adopt a slightly different presentation and treat the general case \eqref{growth.full}. 
We fix a ball $B_{r} \Subset \Omega$ with $r \le 1$,  and consider the homogeneous equation
\begin{equation}\label{homoeq}
-\mathrm{div}\,A(x,Dw)  = 0 \quad \text{in } B_{r}.
\end{equation}

In the next reasoning we will distinguish two cases: for a parameter $\bar{L} \ge 8$, either the $p$-phase
\begin{equation}\label{p.phase.4r}
\sup_{x \in B_{r}}a(x) \le \bar{L}[a]_{\alpha}r^{\alpha}
\end{equation}
or the $(p,q)$-phase
\begin{equation}\label{pq.phase.4r} 
\sup_{x\in B_r}a(x) \ge \bar{L}[a]_{\alpha}r^{\alpha} .
\end{equation}
We shall treat them separately, starting with the $p$-phase \eqref{p.phase.4r}.

\subsubsection{Estimates in the $p$-phase}

We first assume that \eqref{p.phase.4r} holds.
Set
\begin{equation*}
p_0 = \min \left \{ \frac{n\max\{p-1,1\}}{n-1}+ p-q , p \right \}.
\end{equation*}
Observe that $p_0 \leq p$ and that
\[
p_0 \overset{\eqref{rate.sola}}{>} \frac{n}{n-1}\max\{p-1,1\} - \frac{\alpha(p-1)}{n-1} \ge \frac{n-\alpha}{n-1}\max\{p-1,1\} \ge 1.
\]
We then define a function $m:[p_0, p] \to \mathbb{R}$ by
\[ m(t) = (t + q-p)_{*}, \qquad t \in [p_0,p]. \]
Note that the function $[1,n) \ni t \mapsto t^* - t = t^{2}/(n-t)$ is increasing.
Using this, we claim that
\begin{align}\label{m.cont}
m(t) < t \qquad \text{for any}\;\; t \in [p_0, p].
\end{align}
\begin{itemize}
\item[(i)] If $2 \le p \le n$, then (by choosing $t^*$ sufficiently large when $t=n$)
\begin{equation*}
\begin{aligned}
t^* - t & \ge \left(\frac{n(p-1)}{n-1} + p-q \right)^* - \frac{n(p-1)}{n-1} + q-p \\
\overset{\eqref{rate.sola}}& {>} \left(\frac{(n-\alpha)(p-1)}{n-1}\right)^* - \frac{n(p-1)}{n-1} + q-p \\
& \ge (p-1)^* - \frac{n(p-1)}{n-1} + q-p \\
& = \frac{n(p-1)}{n-p+1} - \frac{n(p-1)}{n-1} + q-p \\
& \ge q-p.
\end{aligned}
\end{equation*}
\item[(ii)] If $2-1/n < p < 2$, then
\begin{equation*}
\begin{aligned}
t^{*} - t \ge \left(\frac{n}{n-1} + p-q\right)^{*} - \frac{n}{n-1} + q-p > q-p,
\end{aligned}
\end{equation*}
where in the last inequality we have also used the fact that
\[ \frac{n}{n-1} + p-q \overset{\eqref{rate.sola}}{>} \frac{n}{n-1} - \frac{\alpha(p-1)}{n-1} \ge \frac{n+1-p}{n-1} > 1 \;\; \Longrightarrow \;\; \left(\frac{n}{n-1}+p-q\right)^* > \frac{n}{n-1}. \]
\end{itemize}
Therefore, in any case, we have \eqref{m.cont} as follows:
\[ t^* > t+q-p \;\; \Longleftrightarrow \;\; t > (t+q-p)_* = m(t) \quad \text{for any}\;\; t \in [p_0,p]. \]

We now set the exponent 
\begin{equation}\label{def.kappa}
\kappa \coloneqq \frac{1}{2}\left[ \max\left\{ \frac{n(q-p)}{\alpha}, m(p_0) \right\} + \frac{n(p-1)}{n-1} \right].
\end{equation}
From \eqref{p.bound} and the definition of $p_0$, it directly follows that
\[ m(p_0) \ge \left(\frac{n}{n-1}\right)_* = 1 \;\; \Longrightarrow \;\; \kappa \ge \frac{1}{2}\left( 1 + \frac{n(p-1)}{n-1} \right) >1. \]
Moreover, since
\begin{equation*}
\eqref{rate.sola} \;\; \Longleftrightarrow \;\; \frac{n(q-p)}{\alpha} < \frac{n(p-1)}{n-1}
\end{equation*}
and
\begin{equation}\label{m.bd}
m(p_0) = \left(\frac{n}{n-1}\max\{p-1,1\}\right)_* < \frac{n(p-1)}{n-1},
\end{equation}
we have
\begin{equation}\label{kappa.range} 
\max\left\{ \frac{n(q-p)}{\alpha},m(p_0) \right\} < \kappa < \frac{n(p-1)}{n-1}. 
\end{equation} 
Recall from \eqref{m.cont} that the function $[p_0,p] \ni t \mapsto m(t)/t$ is continuous and bounded above by $1$. Using this fact, together with \eqref{m.cont} and \eqref{m.bd}, we can choose a number $N_{0} \equiv N_{0}(n,p,q,\alpha) \in \mathbb{N} \cup \{ 0 \}$ such that
\[ \underbrace{m\circ m \circ \cdots \circ m}_{N_{0}+1}(p) \leq \kappa < \underbrace{m\circ m \circ \cdots \circ m}_{N_{0}}(p). \]
In this setting, by restricting the domain of $m$ and employing the refined choice of $\kappa$ in \eqref{def.kappa}, we can extend the results in \cite[Lemmas 3.1--3.4]{SY} to the general case \eqref{growth.full}. Consequently, proceeding exactly as in \cite[Section 3]{SY}, we arrive at the following gradient integrability estimate.

\begin{lemma}\label{lem:higher}
Let $w$ be a weak solution to \eqref{homoeq} under assumptions \eqref{growth}, \eqref{a.holder}, \eqref{rate.sola} and \eqref{growth.full}. Assume that $B_{r}$ satisfies \eqref{p.phase.4r} for some $\bar{L} \in [8,\infty)$. Then for any $\tilde{q} < np/(n-2\alpha)$ ($= \infty$ when $n=2$ and $\alpha=1$), $\theta \in (0,\kappa]$ and $\eta \in (0,1)$, we have
\begin{equation}\label{higher.est}
\left(\mean{B_{\eta r}}|Dw|^{\tilde{q}}\,dx\right)^{1/\tilde{q}} \le c\left(\mean{B_{r}}|Dw|^{\theta}\,dx\right)^{1/\theta}
\end{equation}
for a constant $c \equiv c(\data, \|Dw\|_{L^{\kappa}(B_{r})}, \bar{L}, \tilde{q},\theta,\eta)$. 
In particular, it follows that 
\[ Dw \in L^{2q-p}_{\loc}(B_{r}) \subset L^{q}_{\loc}(B_{r}). \] 
\end{lemma}

\subsubsection{Estimates in the $(p,q)$-phase}
We now turn to the complementary case \eqref{pq.phase.4r}, which together with \eqref{a.holder} implies
\[ \sup_{x \in B_r}a(x) \le 2\inf_{x\in B_r}a(x). \]
Therefore, we can see that the leading operator in \eqref{homogeneous} satisfies a certain kind of  Orlicz growth condition in $B_r$.
More precisely, the vector field $A(\cdot)$, initially subject to \eqref{growth}, actually satisfies 
\begin{equation*}
\left\{
\begin{aligned}
|A(x,z)| + |\partial A(x,z)||z| &\le L(|z|^{p-1} + \|a\|_{L^{\infty}(B_{r})}|z|^{q-1}),\\
\frac{\nu}{2} (|z|^{p-2} + \|a\|_{L^{\infty}(B_{r})}|z|^{q-2})|\zeta|^{2} &\le \partial A(x,z) \zeta \cdot \zeta
\end{aligned}
\right.
\end{equation*}
for any $x \in B_r$, $z \in \mathbb{R}^n \setminus \{0\}$ and $\zeta \in \mathbb{R}^{n}$. 
In this case, 
reverse H\"older estimates for \eqref{homogeneous} can be obtained in a rather standard way, i.e., by combining the Caccioppoli type inequality and Sobolev--Poincar\'e type inequality. See \cite[Theorem 9]{DieEtt2008} and also \cite[Section 3.2]{SY}. 
To state the result, we introduce the following notations: 
for the functions $H$ and $h$ given in \eqref{def.H} and a point $x \in \Omega$, we consider two functions $H_{x}, h_{x}: [0,\infty) \to [0,\infty)$ defined by
\begin{equation}\label{def.Hx}
H_{x}(t) \coloneqq H(x,t) \quad \text{and} \quad h_{x}(t) \coloneqq h(x,t) \qquad \text{for}\;\; t \ge 0.
\end{equation}

\begin{lemma}\label{lem.pq.phase}
Let $w$ be a weak solution to \eqref{homogeneous} under assumptions \eqref{growth}, \eqref{a.holder}, \eqref{rate.sola} and \eqref{growth.full}. Assume that $B_r$ satisfies \eqref{pq.phase.4r} for some $\bar{L} \in [8,\infty)$. Then for any $\varepsilon \in (0,1)$ and any $x_* \in B_{r/2}$, there exists a constant $c\equiv c(\data,\bar{L},\varepsilon)$ satisfying
\begin{equation*}
\mean{B_{r/2}}H_{x_*}(|Dw|)\,dx \le c\left(\mean{B_{r}}[H_{x_*}(|Dw|)]^{\varepsilon}\,dx\right)^{1/\varepsilon}.
\end{equation*}
Moreover, by choosing $\varepsilon>0$ sufficiently small in the above estimate, we also have
\begin{equation*}
\mean{B_{r/2}}H_{x_{*}}(|Dw|)\,dx \le c(H_{x_*}\circ h_{x_*}^{-1})\left(\mean{B_r}h_{x_*}(|Du|)\,dx\right)
\end{equation*}
for a constant $c\equiv c(\data,\bar{L})$.
\end{lemma}

\subsubsection{Combining the two phases}
Here we choose $\bar{L} = 8$ and then combine the estimates obtained in each of the cases \eqref{p.phase.4r} and \eqref{pq.phase.4r}. Then we arrive at the following corollary, which extends \cite[Remark 3.6]{SY}:

\begin{corollary}\label{revhol.cor}
Let $w$ be a weak solution to \eqref{homoeq} under assumptions \eqref{growth}, \eqref{a.holder}, \eqref{rate.sola} and \eqref{growth.full}. Then, with the exponent $\kappa$ given in \eqref{def.kappa}, we have
\begin{equation*}
\left(\mean{B_{r/2}}|Dw|^p\,dx\right)^{1/p} \le c\left(\mean{B_r}|Dw|^{\kappa}\,dx\right)^{1/\kappa}
\end{equation*}
for a constant $c\equiv c(n,p,q,\nu,L,\|a\|_{L^{\infty}}, \alpha, [a]_{\alpha}, \|Dw\|_{L^{\kappa}(B_r)})$. 
\end{corollary}

We have fixed the parameter $\bar{L}$ to get Corollary~\ref{revhol.cor} for any ball $B_{r} \Subset \Omega$.
We will use the estimate in the proof of Lemma~\ref{lem.rev.u} which holds regardless of the phases \eqref{p.phase.4r} and \eqref{pq.phase.4r}.
\subsection{Lipschitz regularity for homogeneous frozen equations}
We end this section with a Lipschitz estimate for homogeneous equations with certain kinds of Orlicz growth. Consider a $C^{0}(\mathbb{R}^n) \cap C^{1}(\mathbb{R}^n \setminus \{0\})$-vector field $A_{0}:\mathbb{R}^{n}\rightarrow\mathbb{R}^{n}$ satisfying
\begin{equation}\label{growth.fixed}
\left\{
\begin{aligned}
|A_{0}(z)| + |\partial A_{0}(z)||z| &\le L(|z|^{p-1}+a_{0}|z|^{q-1}), \\
\nu(|z|^{p-2}+a_{0}|z|^{q-2})|\xi|^{2} &\le \partial A_{0}(z)\zeta\cdot\zeta
\end{aligned}
\right.
\end{equation}
for any $z \in \mathbb{R}^n \setminus \{0\}$ and $\zeta \in \mathbb{R}^{n}$, where $a_{0} \ge 0$ is a fixed constant. 
The following local Lipschitz estimate can be found in  \cite{Ba15,Lie91}. 
\begin{lemma}\label{lem.Lip}
Let $v$ be a weak solution to 
\begin{equation*}
-\mathrm{div}\,A_{0}(Dv)=0 \quad \textrm{in}\;\; \Omega
\end{equation*} 
under assumptions \eqref{growth.fixed} with $1 < p < q < \infty$. Then $Dv \in L^{\infty}_{\loc}(\Omega)$. Moreover, there exists a constant $c \equiv c(n,p,q,\nu,L)$, which is in particular independent of $a_{0}$,  such that
\begin{equation*}
\sup_{B_{r}}|Dv| \le c\mean{B_{2r}}|Dv|\,dx
\end{equation*} 
holds whenever $B_{2r} \Subset \Omega$.
\end{lemma}

\section{Comparison estimates}\label{sec.comp}

In this section, we establish comparison estimates between \eqref{main.eq} and a Dirichlet problem involving \eqref{intro.homo}. For this, throughout this section, we assume
\begin{equation}\label{regular.data}
\mu \in L^{\infty}(\Omega), \qquad u \in W^{1,H}_{0}(\Omega). 
\end{equation}
We fix a radius $R_* \equiv R_*(n,p,q,\alpha,|\mu|(\Omega), \|Du\|_{L^{\gamma_0}(\Omega)}) >0$ satisfying
\begin{equation}\label{R0.ini}
\max\left\{ R_{*}^{n-\frac{(n-1)\kappa}{p-1}}[|\mu|(\Omega)]^{\frac{\kappa}{p-1}}, R_{*}^{\alpha-(n-1)\frac{q-p}{p-1}}[|\mu|(\Omega)]^{\frac{q-p}{p-1}}, R_{*}^{\sigma_0}[|\mu|(\Omega)]^{q-p}\|Du\|_{L^{\gamma_0}(\Omega)}^{(2-p)(q-p)} \right\} \le 1,
\end{equation}
where the constants $\gamma_0 \equiv \gamma_0(n,p,q,\alpha) \in (0,n(p-1)/(n-1)) $ and $\sigma_0 \equiv \sigma_0(n,p,q,\alpha)>0$ are given in \eqref{def.gamma0} and \eqref{def.sigma0}, respectively. Note that this is possible due to \eqref{rate.sola}, \eqref{kappa.range} and \eqref{def.sigma0}. 
Later in the proof of Theorem \ref{main.thm}, we will apply the estimates obtained in this section to weak solutions $u_j$ to the regularized problems \eqref{approx.prob} as described in Definition \ref{def:sola}. 
We remark that it is possible to make the radius $R_*$ stable in the approximating procedure, by choosing $j \in \mathbb{N}$ large enough to satisfy
\[ \|Du_j\|_{L^{\gamma_0}(\Omega)} \le 2\|Du\|_{L^{\gamma_0}(\Omega)} \qquad \text{and} \qquad |\mu_j|(\Omega) \le 2\|\mu|(\Omega) \]

We take a ball $B_{4r} \Subset \Omega$ with $4r \le R_*$ and consider a Dirichlet problem
\begin{equation}\label{homogeneous}
\left\{
\begin{aligned}
-\mathrm{div}\,A(x,Dw) & = 0 &\text{in }& B_{2r}, \\
w & = u & \text{on }& \partial B_{2r}.
\end{aligned}
\right.
\end{equation}
Due to \eqref{regular.data}, the existence of a unique weak solution $w \in u + W^{1,H}_{0}(B_{2r})$ to \eqref{homogeneous} can be proved via a monotonicity method in Musielak--Orlicz spaces, see for instance \cite[Lemma~3.4]{BL22}.

We start with a preliminary version of comparison estimate between \eqref{main.eq} and \eqref{homogeneous}. 
\begin{lemma}\label{lem.comp.pre}
Let $u$ be the weak solution to \eqref{main.eq} under assumptions \eqref{growth}--\eqref{rate.sola}, and let $w$ be the weak solution to \eqref{homogeneous}. Then we have 
\begin{equation}\label{comp.pre}
\mean{B_{2r}}|Du-Dw|^{s}\,dx \le c\left[\frac{|\mu|(B_{2r})}{(2r)^{n-1}}\right]^{s/(p-1)} + c\left[\frac{|\mu|(B_{2r})}{(2r)^{n-1}}\right]^{s}\left(\mean{B_{2r}}|Du|^{s}\,dx\right)^{2-p}
\end{equation}
for any $s$ satisfying \eqref{exp.range}, where $c\equiv c(n,p,\nu,L,s)$. 
\end{lemma}
\begin{proof}
Let $d>0$ and $\xi>1$ be any numbers. Taking $\varphi_{\pm} \coloneqq d^{1-\xi} -(d+(u-w)_{\pm})^{1-\xi}$ as a test function in the weak formulation of
\[ -\mathrm{div}\,(A(x,Du)-A(x,Dw)) = \mu \quad \text{in } B_{2r}, \]
we obtain
\begin{equation*}
\begin{aligned}
& \left| (\xi-1)\int_{B_{2r}}\frac{(A(x,Du)-A(x,Dw))\cdot D(u-w)_{\pm}}{(d+(u-w)_{\pm})^{\xi}}\,dx \right| \\
& = \left| \int_{B_{2r}}(A(x,Du)-A(x,Dw))\cdot D\varphi_{\pm} \,dx \right|  \\
& = \left| \int_{B_{2r}}\varphi_{\pm}\,d\mu \right| \le d^{1-\xi}|\mu|(B_{2r}).
\end{aligned}
\end{equation*}
Thus, \eqref{mono} and an elementary manipulation lead to
\begin{align}\label{weighted.comparison}
\int_{B_{2r}} \frac{|V_{p}(Du)-V_{p}(Dw)|^2 + a(x)|V_{q}(Du)-V_{q}(Dw)|^{2}}{(d+|u-w|)^\xi} \,dx \le c\frac{d^{1-\xi}}{\xi-1}|\mu|(B_{2r})
\end{align}
for some $c\equiv c(n,p,q,\nu,L)$. In particular, we have
\begin{align*}
\int_{B_{2r}} \frac{|V_{p}(Du)-V_{p}(Dw)|^2}{(d+|u-w|)^\xi} \,dx \leq c\frac{d^{1-\xi}}{\xi-1}|\mu|(B_{2r}).
\end{align*}
From this estimate, we can derive \eqref{comp.pre} in a rather standard way, see for instance \cite[Lemma 4.1]{Ba17}.
\end{proof}

The following lemma is concerned with a reverse H\"older type estimate for $Du$.
\begin{lemma}\label{lem.rev.u}
Let $u$ be the weak solution to \eqref{main.eq} under assumptions \eqref{growth}--\eqref{rate.sola}, and let $B_r \Subset \Omega$ be a ball with $r \le R_*$. Then we have
\begin{equation*}
\left(\mean{B_{\eta r}}|Du|^{s}\,dx\right)^{1/s} \le c\left(\mean{B_{r}}|Du|^{\theta}\,dx\right)^{1/\theta} + c\left[\frac{|\mu|(B_r)}{r^{n-1}}\right]^{1/(p-1)}
\end{equation*}
for any $s$ satisfying \eqref{exp.range}, $\theta \in (0,s]$ and $\eta \in (0,1)$, where $c\equiv c(\data, \|Du\|_{L^{\kappa}(\Omega)}, s, \theta, \eta)$.
\end{lemma}
\begin{proof}
We fix $\eta \le \sigma < \tau \le 1$, and consider the homogeneous problem
\begin{equation*}
\left\{
\begin{aligned}
-\mathrm{div}\,A(x,Dw_\tau) &=0 &\text{in }& B_{\tau r}, \\
w_{\tau} &= u &\text{on }& \partial B_{\tau r}.
\end{aligned}
\right.
\end{equation*}
First, using \eqref{comp.pre} and Young's inequality, we have
\begin{equation}\label{comp.young}
\left(\mean{B_{\tau r}}|Du-Dw_{\tau}|^{s}\,dx\right)^{1/s} \le c\varepsilon^{\frac{p-2}{p-1}}\left[\frac{|\mu|(B_r)}{r^{n-1}}\right]^{1/(p-1)} + \varepsilon\left(\mean{B_{\tau r}}|Du|^{s}\,dx\right)^{1/s}
\end{equation}
for any $s$ satisfying \eqref{exp.range} and $\varepsilon \in (0,1]$, where $c\equiv c(n,p,\nu,L,s)$. 
We next recall reverse H\"older type estimates for $Dw_\tau$. Observe that, by using Corollary \ref{revhol.cor} and a standard covering argument, we can obtain
\begin{equation}\label{revhol.cover}
\left(\mean{B_{\sigma r}}|Dw_{\tau}|^{p}\,dx\right)^{1/p} \le \frac{c}{(\tau-\sigma)^{n(1/\theta - 1/p)}}\left(\mean{B_{\tau r}}|Dw_{\tau}|^{\theta}\,dx\right)^{1/\theta}
\end{equation}
for any $\theta \in (0,s]$, where $c\equiv c(\data, \|Dw_{\tau}\|_{L^{\kappa}(B_{\tau r})}, \theta,\eta)$ and $\kappa$ is given in \eqref{def.kappa}. 
We further estimate the quantity $\|Dw_{\tau}\|_{L^{\kappa}(B_{\tau r})}$. In light of \eqref{kappa.range}, we can use \eqref{comp.young} (with $\varepsilon=1$) and then \eqref{R0.ini} in order to get
\begin{equation*}
\begin{aligned}
\|Dw_{\tau}\|_{L^{\kappa}(B_{\tau r})} & \le \|Dw_{\tau}-Du\|_{L^{\kappa}(B_{\tau r})} + \|Du\|_{L^{\kappa}(B_{\tau r})} \\
& \le cr^{n/\kappa - (n-1)/(p-1)}[|\mu|(B_r)]^{1/(p-1)} + 2\|Du\|_{L^{\kappa}(B_{r})} \\
& \le c(\data, \|Du\|_{L^{\kappa}(\Omega)},\eta).
\end{aligned}
\end{equation*}
We now proceed as
\begin{equation*}
\begin{aligned}
& \left(\mean{B_{\sigma r}}|Du|^{s}\,dx\right)^{1/s} \le \left(\mean{B_{\sigma r}}|Du-Dw_{\tau}|^{s}\,dx\right)^{1/s} + \left(\mean{B_{\sigma r}}|Dw_{\tau}|^{s}\,dx\right)^{1/s} \\
\overset{\eqref{revhol.cover}}&{\le} \left(\mean{B_{\sigma r}}|Du-Dw_{\tau}|^{s}\,dx\right)^{1/s} + \frac{c}{(\tau-\sigma)^{n(1/\theta - 1/p)}}\left(\mean{B_{\tau r}}|Dw_{\tau}|^{\theta}\,dx\right)^{1/\theta} \\
& \le \left(\mean{B_{\sigma r}}|Du-Dw_{\tau}|^{s}\,dx\right)^{1/s} + \frac{c}{(\tau-\sigma)^{n(1/\theta - 1/p)}}\left(\mean{B_{\tau r}}|Du|^{\theta}\,dx + \mean{B_{\tau r}}|Du-Dw_{\tau}|^{\theta}\,dx\right)^{1/\theta} \\
& \le \frac{c}{(\tau-\sigma)^{n(1/\theta-1/p)}}\left[ \left(\mean{B_{\tau r}}|Du-Dw_{\tau}|^{s}\,dx\right)^{1/s} + \left(\mean{B_{r}}|Du|^{\theta}\,dx\right)^{1/\theta} \right] \\
\overset{\eqref{comp.young}}&{\le} \frac{c}{(\tau-\sigma)^{n(1/\theta - 1/p)}}\left\{ \varepsilon^{\frac{p-2}{p-1}}\left[\frac{|\mu|(B_r)}{r^{n-1}}\right]^{1/(p-1)} + \varepsilon\left(\mean{B_{\tau r}}|Du|^{s}\,dx\right)^{1/s} + \left(\mean{B_{r}}|Du|^{\theta}\,dx\right)^{1/\theta} \right\} 
\end{aligned}
\end{equation*}
for any $\varepsilon \in (0,1]$, where $c\equiv c(\data, \|Du\|_{L^{\kappa}(\Omega)}, s,\theta,\eta)$. 
Choosing $\varepsilon$ small enough, we arrive at
\begin{equation*}
\begin{aligned}
& \left(\mean{B_{\sigma r}}|Du|^{s}\,dx\right)^{1/s} \\
& \le \frac{1}{2}\left(\mean{B_{\tau r}}|Du|^{s}\,dx\right)^{1/s} + \frac{c}{(\tau-\sigma)^{\beta}}\left\{ \left(\mean{B_{r}}|Du|^{\theta}\,dx\right)^{1/\theta} + \left[\frac{|\mu|(B_r)}{r^{n-1}}\right]^{1/(p-1)} \right\}
\end{aligned}
\end{equation*}
whenever $\eta \le \sigma < \tau \le 1$, where $c\equiv c(\data, \|Du\|_{L^{\kappa}(\Omega)}, s,\theta,\eta)$ and $\beta \equiv \beta(n,p,\theta)$ are positive constants. 
Finally, an application of Lemma \ref{tech.lemma} leads to the conclusion.
\end{proof}

Using the above lemma (with $\eta = 2/3$ and $B_r$ replaced by $B_{3r}$) and H\"older's inequality, we can improve Lemma \ref{lem.comp.pre} as follows.
\begin{lemma}\label{improved.comp}
Let $u$ be the weak solution to \eqref{main.eq} under assumptions \eqref{growth}--\eqref{rate.sola}, and let $w$ be the weak solution to \eqref{homogeneous}. Then we have
\begin{equation*}
\mean{B_{2r}}|Du-Dw|^{s}\,dx \le c\left[\frac{|\mu|(B_{3r})}{(3r)^{n-1}}\right]^{s/(p-1)} + c\left[\frac{|\mu|(B_{3r})}{(3r)^{n-1}}\right]^{s}\left(\mean{B_{3r}}|Du|^{\theta}\,dx\right)^{(2-p)s/\theta}
\end{equation*}
for any 
\begin{equation*}
 s, \theta \in \left(0, \frac{n(p-1)}{n-1}\right),
\end{equation*}
where $c\equiv c(\data, \|Du\|_{L^{\kappa}(\Omega)},s,\theta)$.
\end{lemma}

We now establish a comparison estimate involving the function $h$ given in \eqref{def.H}.

\begin{lemma}\label{lem.1st.h}
Let $u$ be the weak solution to \eqref{main.eq} under assumptions \eqref{growth}--\eqref{rate.sola}, and let $w$ be the weak solution to \eqref{homogeneous}. Then we have
\begin{equation}\label{1st.comparison.h}
\begin{aligned}
\mean{B_{2r}}h(x,|Du-Dw|)\,dx & \le c\left[\frac{|\mu|(B_{4r})}{(4r)^{n-1}}\right] + c\left[\frac{|\mu|(B_{4r})}{(4r)^{n-1}}\right]^{p-1}\left(\mean{B_{4r}}|Du|^{p-1}\,dx\right)^{2-p} \\
& \quad\; + c\chi_{\{q<2\}}\left[\frac{|\mu|(B_{4r})}{(4r)^{n-1}}\right]^{q-1}\left(\mean{B_{4r}}a(x)|Du|^{q-1}\,dx\right)^{2-q}
\end{aligned}
\end{equation}
for a constant $c\equiv c(\data, \|Du\|_{L^{\kappa}(\Omega)})$.
\end{lemma}
\begin{proof}
Lemma \ref{improved.comp} implies
\begin{equation}\label{compest.p-term}
\mean{B_{2r}}|Du-Dw|^{p-1}\,dx \le c\left[\frac{|\mu|(B_{3r})}{(3r)^{n-1}}\right] + c\left[\frac{|\mu|(B_{3r})}{(3r)^{n-1}}\right]^{p-1}\left(\mean{B_{3r}}|Du|^{p-1}\,dx\right)^{2-p}
\end{equation}
for some $c\equiv c(\data, \|Du\|_{L^{\kappa}(\Omega)})$. To proceed further, we consider the following two cases separately:
\begin{equation}\label{alt.sec4}
\inf_{x \in B_{4r}}a(x) > 8[a]_{\alpha}r^{\alpha} \qquad \text{and} \qquad \inf_{x \in B_{4r}}a(x) \le 8[a]_{\alpha}r^{\alpha}.
\end{equation}

\textit{The case $\eqref{alt.sec4}_1$. } 
In this case, observe from \eqref{a.holder} and $\eqref{alt.sec4}_1$ that
\begin{equation}\label{a.pqp}  
8[a]_{\alpha}r^{\alpha} < \inf_{x \in B_{4r}}a(x) \le \sup_{x \in B_{4r}}a(x) \le \inf_{x \in B_{4r}}a(x) + [a]_{\alpha}(8r)^{\alpha} \le 2\inf_{x \in B_{4r}}a(x). 
\end{equation}
Accordingly, letting $\tilde{\mu} \coloneqq \mu/\|a\|_{L^{\infty}(B_{4r})}$, \eqref{weighted.comparison} implies
\begin{equation*}
\int_{B_{2r}}\frac{|V_{q}(Du)-V_{q}(Dw)|^2}{(d+|u-w|)^{\xi}}\,dx \le c\frac{d^{1-\xi}}{\xi-1}|\tilde{\mu}|(B_{2r})
\end{equation*}
for some $c\equiv c(n,p,q,\nu,L)$. Then we can proceed similarly to the proofs of \cite[Lemma~2]{KM14BMS} (when $q \ge 2$) and \cite[Lemma 4.1]{Ba17} (when $2-1/n < q < 2$), thereby obtaining
\begin{equation*}
\begin{aligned}
& \mean{B_{2r}}|Du-Dw|^{q-1}\,dx \\
& \le c\left[\frac{|\tilde{\mu}|(B_{3r})}{(3r)^{n-1}}\right] + c\chi_{\{q<2\}}\left[\frac{|\tilde{\mu}|(B_{3r})}{(3r)^{n-1}}\right]^{q-1}\left(\mean{B_{3r}}|Du|^{q-1}\,dx\right)^{2-q} \\
& = c\|a\|_{L^{\infty}(B_{4r})}^{-1}\left[\frac{|\mu|(B_{3r})}{(3r)^{n-1}}\right] + c\chi_{\{q<2\}}\|a\|_{L^{\infty}(B_{4r})}^{1-q}\left[\frac{|\mu|(B_{3r})}{(3r)^{n-1}}\right]^{q-1}\left(\mean{B_{3r}}|Du|^{q-1}\,dx\right)^{2-q}
\end{aligned}
\end{equation*}
for some $c\equiv c(n,p,q,\nu,L)$; note that when $2-1/n < q < 2$, we have also used Lemma \ref{lem.rev.u} with $\theta = q-1$, $\eta = 2/3$ and $B_r$ replaced by $B_{3r}$. 
We thus obtain, using \eqref{a.pqp} twice,
\begin{equation*}
\begin{aligned}
& \mean{B_{2r}}a(x)|Du-Dw|^{q-1}\,dx \le \|a\|_{L^{\infty}(B_{4r})}\mean{B_{2r}}|Du-Dw|^{q-1}\,dx \\
& \le c\left[\frac{|\mu|(B_{3r})}{(3r)^{n-1}}\right] + c\chi_{\{q<2\}}\|a\|_{L^{\infty}(B_{4r})}^{2-q}\left[\frac{|\mu|(B_{3r})}{(3r)^{n-1}}\right]^{q-1}\left(\mean{B_{3r}}|Du|^{q-1}\,dx\right)^{2-q} \\
& \le c\left[\frac{|\mu|(B_{3r})}{(3r)^{n-1}}\right] + c\chi_{\{q<2\}}\left[\frac{|\mu|(B_{3r})}{(3r)^{n-1}}\right]^{q-1}\left(\mean{B_{3r}}a(x)|Du|^{q-1}\,dx\right)^{2-q}
\end{aligned}
\end{equation*}
for some $c\equiv c(\data, \|Du\|_{L^{\kappa}(\Omega)})$.
Combining this estimate and \eqref{compest.p-term}, we obtain \eqref{1st.comparison.h}.

\textit{The case $\eqref{alt.sec4}_2$. } 
In this case,  \eqref{a.holder} and $\eqref{alt.sec4}_2$ imply
\begin{equation}\label{a.pp} 
\sup_{x \in B_{4r}}a(x) \le \inf_{x \in B_{4r}}a(x) + [a]_{\alpha}(8r)^{\alpha} \le 16[a]_{\alpha}r^{\alpha}. 
\end{equation}
We set the exponent
\begin{equation}\label{def.gamma0}
\gamma_{0} \coloneqq \frac{1}{2}\left(\frac{n(2-p)}{\alpha/(q-p) -(n-1)} + \frac{n(p-1)}{n-1}\right). 
\end{equation}
Observe that
\begin{equation}\label{gamma0.range}
\eqref{rate.sola} \;\; \Longrightarrow \;\; \frac{n(2-p)}{\alpha/(q-p)-(n-1)} < \gamma_{0} <\frac{n(p-1)}{n-1}.
\end{equation}
This and \eqref{q-1.admissible} allow us to use Lemma \ref{improved.comp} with $s = q-1$ and $\theta = \gamma_0$, which gives
\begin{equation}\label{est.pp.1st}
\begin{aligned}
& \mean{B_{2r}}a(x)|Du-Dw|^{q-1} \overset{\eqref{a.pp}}{\le} cr^{\alpha}\mean{B_{2r}}|Du-Dw|^{q-1}\,dx \\
& \le cr^{\alpha}\left[\frac{|\mu|(B_{3r})}{(3r)^{n-1}}\right]^{\frac{q-1}{p-1}} + cr^{\alpha}\left[\frac{|\mu|(B_{3r})}{(3r)^{n-1}}\right]^{q-1}\left(\mean{B_{3r}}|Du|^{\gamma_0}\,dx\right)^{(2-p)(q-1)/\gamma_0}
\end{aligned}
\end{equation}
for a constant $c\equiv c(\data, \|Du\|_{L^{\kappa}(\Omega)})$. We further estimate each term in the right-hand side as
\begin{equation}\label{est.measure.pp}
\begin{aligned}
r^{\alpha}\left[\frac{|\mu|(B_{3r})}{(3r)^{n-1}}\right]^{\frac{q-1}{p-1}} & = r^{\alpha}\left[\frac{|\mu|(B_{3r})}{(3r)^{n-1}}\right]^{\frac{q-p}{p-1} + 1} \\
& \le cr^{\alpha - \frac{q-p}{p-1}(n-1)}[|\mu|(\Omega)]^{\frac{q-p}{p-1}}\left[\frac{|\mu|(B_{3r})}{(3r)^{n-1}}\right] \\
\overset{\eqref{R0.ini}}&{\le} c\left[\frac{|\mu|(B_{3r})}{(3r)^{n-1}}\right]
\end{aligned}
\end{equation}
and
\begin{equation*}
\begin{aligned}
& r^{\alpha}\left[\frac{|\mu|(B_{3r})}{(3r)^{n-1}}\right]^{q-1}\left(\mean{B_{3r}}|Du|^{\gamma_0}\,dx\right)^{(2-p)(q-1)/\gamma_0} \\
& = r^{\alpha}\left[\frac{|\mu|(B_{3r})}{(3r)^{n-1}}\right]^{q-p + p-1}\left(\mean{B_{3r}}|Du|^{\gamma_0}\,dx\right)^{(2-p)(q-p)/\gamma_0 + (2-p)(p-1)/\gamma_0} \\
& \le cr^{\sigma_0}[|\mu|(\Omega)]^{q-p}\|Du\|_{L^{\gamma_0}(\Omega)}^{(2-p)(q-p)}\left[\frac{|\mu|(B_{3r})}{(3r)^{n-1}}\right]^{p-1}\left(\mean{B_{3r}}|Du|^{\gamma_0}\,dx\right)^{(2-p)(p-1)/\gamma_0},
\end{aligned}
\end{equation*}
where $c\equiv c(\data, \|Du\|_{L^{\kappa}(\Omega)})$ and
\begin{equation}\label{def.sigma0}
\sigma_0 \coloneqq \alpha-(q-p)(n-1)-\frac{n(2-p)(q-p)}{\gamma_0} > 0 \quad \text{by}\;\; \eqref{gamma0.range}. 
\end{equation}
In light of \eqref{R0.ini}, we estimate
\begin{equation}\label{est.Du.pp}
\begin{aligned}
& r^{\alpha}\left[\frac{|\mu|(B_{3r})}{(3r)^{n-1}}\right]^{q-1}\left(\mean{B_{3r}}|Du|^{\gamma_0}\,dx\right)^{(2-p)(q-1)/\gamma_0} \\
& \le c\left[\frac{|\mu|(B_{3r})}{(3r)^{n-1}}\right]^{p-1}\left(\mean{B_{3r}}|Du|^{\gamma_0}\,dx\right)^{(2-p)(p-1)/\gamma_0} \\
& \le c\left[\frac{|\mu|(B_{4r})}{(4r)^{n-1}}\right] + c\left[\frac{|\mu|(B_{4r})}{(4r)^{n-1}}\right]^{p-1}\left(\mean{B_{4r}}|Du|^{p-1}\,dx\right)^{2-p}
\end{aligned}
\end{equation}
for some $c\equiv c(\data,\|Du\|_{L^{\kappa}(\Omega)})$, where in the last line we also have used Lemma \ref{lem.rev.u} (with $\theta = p-1$, $\eta = 3/4$ and $B_r$ replaced by $B_{4r}$). 
Connecting \eqref{est.measure.pp} and \eqref{est.Du.pp} to \eqref{est.pp.1st}, we arrive at
\begin{equation*}
\mean{B_{2r}}a(x)|Du-Dw|^{q-1}\,dx \le c\left[\frac{|\mu|(B_{4r})}{(4r)^{n-1}}\right] + c\left[\frac{|\mu|(B_{4r})}{(4r)^{n-1}}\right]^{p-1}\left(\mean{B_{4r}}|Du|^{p-1}\,dx\right)^{2-p},
\end{equation*}
which together with \eqref{compest.p-term} leads to \eqref{1st.comparison.h}.
\end{proof}

\section{Proof of the main theorem}\label{sec4}

Now we consider the general case when $\mu$ is a signed Borel measure with finite total mass.
We start this section with recalling the definition of $\mathbf{M}_1(\mu)$  in \eqref{def.M1}; observe that there exists a constant  $c_{0} \equiv c_{0}(n) \ge 1$ such that
\begin{equation}\label{ineq.M1}
\frac{|\mu|(\overline{B_{r}(x_{0})})}{r^{n-1}}  \le c_{0}\mean{B_{r}(x_{0})}\left[\frac{|\mu|(B_{2r}(x))}{(2r)^{n-1}}\right]\,dx 
 \le c_{0}\mean{B_{r}(x_{0})}\mathbf{M}_{1}(\mu)\,dx
\end{equation}
holds whenever $B_{r}(x_{0}) \subset \mathbb{R}^{n}$ is a ball.

We now prove our main theorem.

\subsection{Proof of Theorem~\ref{main.thm}}
The proof is divided into eleven steps.

\textit{Step 1: Exit time and covering arguments. } 
We take a ball $B_{2R}  \Subset \Omega$ with $2R \le R_0$ as in the statement; we initially assume that $R_0 \le R_*$ for the radius $R_*$ chosen in \eqref{R0.ini}. 
The value of $R_0$ will be determined later in the proof. 

For any ball $B_{\varrho}(x_0) \subset B_{2R}$ and $\lambda>0$, we consider the upper level sets
\begin{equation*}
E(B_{\varrho}(x_0),\lambda) \coloneqq \{x \in B_{\varrho}(x_0): h(x,|Du(x)|)>\lambda\}
\end{equation*}
and
\[ \mathcal{E}(B_{\varrho}(x_0),\lambda) \coloneqq \{x \in B_{\varrho}(x_0): \mathbf{M}_{1}(\mu)(x) > \lambda \}.  \]
With $M \ge 1$ being a free parameter, whose value will be eventually determined in the end of the proof, we define
\begin{equation*}
\Psi(B_{\varrho}(x_{0})) \coloneqq \mean{B_{\varrho}(x_{0})}\left[h(x,|Du|)+M\mathbf{M}_{1}(\mu)\right]\,dx.
\end{equation*}
again for any ball $B_{\varrho}(x_0) \subset B_{2R}$. It is readily seen that
\begin{equation}\label{psi.limit}
\lim_{\varrho\searrow0}\Psi(B_{\varrho}(x_{0}))> \lambda \qquad \text{for a.e.} \;\; x_{0} \in E(B_s,\lambda) \;\; \text{and} \;\; R \le s \le 2R.
\end{equation}
We choose two radii $r_1,r_2$ such that $R \le r_1 < r_2 \le 2R$. Observe that if $x_{0} \in B_{r_{1}}$, then for any $\varrho \in [(r_{2}-r_{1})/20, r_{2}-r_{1}]$ we have
\begin{equation}\label{def.lambda0}
\Psi(B_{\varrho}(x_{0})) \le \frac{20^{n}r_{2}^{n}}{(r_{2}-r_{1})^{n}}\mean{B_{r_{2}}}\left[h(x,|Du|) + M\mathbf{M}_{1}(\mu)\right]\,dx \eqqcolon \lambda_{0}.
\end{equation}

From now on, we consider 
\begin{equation}\label{lambda.range}
\lambda > \lambda_0. 
\end{equation}
In light of \eqref{psi.limit} and \eqref{def.lambda0}, there exists an exit time radius $\varrho_{x_{0}} \in (0, (r_{2}-r_{1})/20)$ satisfying
\begin{equation*}
\Psi(B_{\varrho_{x_{0}}}(x_{0})) = \lambda \qquad \text{and} \qquad \Psi(B_{\varrho}(x_{0})) < \lambda \quad \text{for every } \varrho \in (\varrho_{x_{0}},r_{2}-r_{1}].
\end{equation*} 
Since this is valid for a.e. $x_0 \in E(B_{r_1},\lambda)$, 
the family $\{B_{\varrho_{x_{0}}}(x_{0}) : x_{0} \in E(B_{r_{1}},\lambda)\}$ covers $E(B_{r_1},\lambda)$ up to a negligible set. Thus, Vitali's covering lemma implies that there exists a countable family $\{B_{\varrho_{x_i}}(x_i)\} \equiv \{\tilde{B}_{i}\}$ of mutually disjoint balls satisfying
\begin{equation}\label{covering}
E(B_{r_1} ,\lambda) \subset \bigcup_{i}5\tilde{B}_{i} \cup \text{negligible set} 
\end{equation}
and, whenever $i \in \mathbb{N}$,
\begin{equation}\label{exit.radius}
\Psi(B_{\varrho_{x_i}}(x_i)) = \lambda \qquad \text{and} \qquad \Psi(B_{\varrho}(x_i))<\lambda \quad \text{for every}\;\; \varrho \in (\varrho_{x_i}, r_{2}-r_{1}].
\end{equation}
In the following, we simply denote
\begin{equation}\label{def.Bi} 
B_{i} \coloneqq 5B_{\varrho_{x_i}}(x_i) = 5\tilde{B}_{i} \qquad \text{and} \qquad \varrho_i \coloneqq 5\varrho_{x_i}. 
\end{equation}
Observe that
\begin{equation*}
\varrho_{i} = 5\varrho_{x_i} \le \frac{r_{2}-r_{1}}{4} \le R_0 \le 1
\end{equation*}
and that, by construction, 
\begin{equation*}
40\tilde{B}_{i} = 8B_{i} \subset B_{r_2}.
\end{equation*}
We also note that \eqref{exit.radius} implies
\begin{equation}\label{exit2}
\left\{
\begin{aligned}
 \Psi(\tilde{B}_{i}) & = \mean{\tilde{B}_{i}}\left[ h(x,|Du|) + M\mathbf{M}_{1}(\mu)\right]\,dx = \lambda, \\
 \Psi(8B_{i}) & = \mean{8B_{i}}\left[ h(x,|Du|) + M\mathbf{M}_{1}(\mu)\right]\,dx \le \lambda.
\end{aligned}
\right.
\end{equation}

\textit{Step 2: Approximation and a first comparison function. }
With $B_i$ being any fixed ball considered in the above, we choose a sequence of functions $\{\mu_j\} \subset L^{\infty}(\Omega)$ such that $\mu_{j}\rightharpoonup \mu$ weakly* in the sense of measures and \eqref{muj.conv} is in force; note that such a sequence can be constructed via mollification. Accordingly, we choose a sequence of corresponding weak solutions $\{u_{j}\} \subset W^{1,H}_{0}(\Omega)$ to \eqref{approx.prob} such that $u_j \to u$ in $W^{1,s}_{0}(\Omega)$ for any $s$ satisfying \eqref{exp.range}. 
In particular, we can choose $j \in \mathbb{N}$ large enough to satisfy
\begin{equation}\label{approx.Du.2}
\|Du_j \|_{L^{\kappa}(\Omega)}  \le 2 \|Du \|_{L^{\kappa}(\Omega)},
\end{equation}
\begin{equation}\label{approx.Du}
\mean{8B_i}h(x,|Du_{j}-Du|)\,dx \le  \frac{\lambda_0}{M}
\end{equation}
and
\begin{equation}\label{approx.mu}
|\mu_j|(8B_i) \le 2|\mu|(\overline{8B_i}).
\end{equation}
Here, $\kappa$ is the constant  defined in \eqref{def.kappa}.
We next consider the weak solution $w_{i,j} \in u_{j} + W^{1,H}_{0}(4B_i)$ to the Dirichlet problem
\begin{equation*}
\left\{
\begin{aligned}
-\mathrm{div}\, A(x,Dw_{i,j})& = 0&\text{in }& 4B_{i}, \\
w_{i,j} & = u_{j}& \text{on }& \partial (4B_i).
\end{aligned}
\right.
\end{equation*}
Lemma \ref{lem.1st.h} and Young's inequality imply
\begin{equation}\label{comp1}
\begin{aligned}
\mean{4B_{i}}h(x,|Du_{j}-Dw_{i,j}|)\,dx & \le c\left[\frac{|\mu|(\overline{8B_{i}})}{(8\varrho_i)^{n-1}}\right] + c\left[\frac{|\mu|(\overline{8B_{i}})}{(8\varrho_i)^{n-1}}\right]^{p-1}\left(\mean{8B_{i}}|Du_j|^{p-1}\,dx\right)^{2-p} \\
& \quad\; + c\chi_{\{q<2\}}\left[\frac{|\mu|(\overline{8B_{i}})}{(8\varrho_i)^{n-1}}\right]^{q-1}\left(\mean{8B_{i}}a(x)|Du_j|^{q-1}\,dx\right)^{2-q} \\
&\le c\varepsilon_{0}^{\frac{p-2}{p-1}}\left[\frac{|\mu|(\overline{8B_{i}})}{(8\varrho_i)^{n-1}}\right] 
+ \varepsilon_{0}\mean{8B_i}h(x,|Du_j|)\,dx
\end{aligned}
\end{equation}
for a constant $c\equiv c(\data, \|Du\|_{L^{\kappa}(\Omega)})$, whenever $\varepsilon_0 \in (0,1]$. Note that we also have used \eqref{approx.Du.2} and \eqref{approx.mu}.

\textit{Step 3: A second comparison function. } 
With the same ball $4B_{i}$ as in the previous step, we choose a point $\tilde{x}_{i} \in \overline{2B_{i}}$ satisfying
\[ a(\tilde{x}_{i}) =  \sup_{x \in 2B_{i}}a(x). \]
We then consider the following homogeneous frozen problem: 
\begin{equation}\label{def.v}
\left\{
\begin{aligned}
-\mathrm{div}\, A(\tilde{x}_{i},Dv_{i,j}) &= 0 &\textrm{in }& 2B_{i}, \\
v_{i,j} &= w_{i,j} &\textrm{on }& \partial (2B_{i}).
\end{aligned}
\right.
\end{equation}
Since $w_{i,j} \in W^{1,q}(2B_i)$ by Lemma~\ref{lem:higher}, 
the existence of a unique weak solution $v_{i,j}$ to \eqref{def.v} follows from standard monotonicity methods in Orlicz spaces. Moreover, $v_{i,j}$ satisfies the energy estimate
\begin{equation} \label{v.energy}
\mean{2B_{i}}(|Dv_{i,j}|^{p}+a(\tilde{x}_{i})|Dv_{i,j}|^{q})\,dx \le c\mean{2B_{i}}(|Dw_{i,j}|^{p}+a(\tilde{x}_{i})|Dw_{i,j}|^{q})\,dx 
\end{equation}
for some $c\equiv c(n,p,q,\nu,L)$.

Proceeding exactly as in the proof of \cite[(4.26)]{CM16JFA}, we obtain
\begin{equation}\label{comp.I}
\begin{aligned}
& \mean{2B_i}\left(|V_{p}(Dw_{i,j})-V_{p}(Dv_{i,j})|^{2}+a(\tilde{x}_{i})|V_{q}(Dw_{i,j})-V_{q}(Dv_{i,j})|^{2}\right)\,dx  \\
& \le c\left(\osc_{2B_i}a\right)\mean{2B_i}|Dw_{i,j}|^{q-1}|Dv_{i,j}-Dw_{i,j}|\,dx \eqqcolon I
\end{aligned}
\end{equation}
for some $c\equiv c(n,p,q,\nu,L)$.

\textit{Step 4: Two different phases. } 
In order to estimate $I$, we divide the cases as follows:
\begin{equation}\label{alt.comp2}
\inf_{x\in 2B_i}a(x) > K[a]_{\alpha}\varrho_{i}^{\alpha} \qquad \text{and} \qquad \inf_{x \in 2B_i}a(x) \le K[a]_{\alpha}\varrho_{i}^{\alpha},
\end{equation}
where $K \ge 4$ is a free parameter whose value will be determined later in the proof.

\textit{Step 5: Estimates in the $(p,q)$-phase. } Here we consider the case $\eqref{alt.comp2}_{1}$. 
In this case, we have 
\[ \osc_{2B_{i}}a \le [a]_{\alpha}(4\varrho_i)^{\alpha} \le \frac{4}{K}a(x) \] 
and 
\begin{equation}\label{a.equiv}
a(\tilde{x}_{i}) \le a(x) + [a]_{\alpha}(2\varrho_i)^{\alpha} \le a(x) + \frac{4}{K}a(x) \le 2a(x)
\end{equation}
 for any $x\in 2B_{i}$. Using these, we estimate $I$ as
\begin{equation*}
\begin{aligned}
I & \le \frac{c}{K}\mean{2B_{i}}a(x)|Dw_{i,j}|^{q-1}|Dv_{i,j}-Dw_{i,j}|\,dx \\
& \le \frac{c}{K}\mean{2B_{i}}a(x)(|Dv_{i,j}|+|Dw_{i,j}|)^{q}\,dx 
\overset{\eqref{v.energy}} {\le} \frac{c}{K}\mean{2B_{i}}(|Dw_{i,j}|^{p}+a(\tilde{x}_{i})|Dw_{i,j}|^{q})\,dx.
\end{aligned}
\end{equation*}
Plugging this inequality into \eqref{comp.I} yields
\begin{equation}\label{pq.phase.I}
\begin{aligned}
& \mean{2B_i}\left(|V_{p}(Dw_{i,j})-V_{p}(Dv_{i,j})|^{2}+a(\tilde{x}_{i})|V_{q}(Dw_{i,j})-V_{q}(Dv_{i,j})|^{2}\right)\,dx  \\
& \le \frac{c}{K}\mean{2B_{i}}(|Dw_{i,j}|^{p}+a(\tilde{x}_{i})|Dw_{i,j}|^{q})\,dx 
\end{aligned}
\end{equation}
for some $c\equiv c(n,p,q,\nu,L)$; we accordingly get
\begin{equation*}
\begin{aligned}
& \mean{2B_i}|Dw_{i,j}-Dv_{i,j}|^{p}\,dx \overset{\eqref{VV}}{\le} c\mean{2B_i}|V_{p}(Dw_{i,j}) - V_{p}(Dv_{i,j})|^{p}(|Dw_{i,j}|+|Dv_{i,j}|)^{p(2-p)/2}\,dx \\
& \le c\left(\mean{2B_i}|V_{p}(Dw_{i,j})-V_{p}(Dv_{i,j})|^{2}\,dx\right)^{p/2}\left(\mean{2B_i}(|Dw_{i,j}|+|Dv_{i,j}|)^{p}\,dx\right)^{(2-p)/2} \\
\overset{\eqref{v.energy},\eqref{pq.phase.I}}&{\le} \frac{c}{K^{p/2}}\mean{2B_i}(|Dw_{i,j}|^{p}+a(\tilde{x}_{i})|Dw_{i,j}|^{q})\,dx.
\end{aligned}
\end{equation*}
When $q<2$, a completely similar calculation also gives
\begin{equation*}
\mean{2B_i}a(\tilde{x}_{i})|Dw_{i,j}-Dv_{i,j}|^{q}\,dx \le \frac{c}{K^{q/2}}\mean{2B_i}(|Dw_{i,j}|^{p}+a(\tilde{x}_{i})|Dw_{i,j}|^{q})\,dx.
\end{equation*}
When $q \ge 2$,  we directly have
\begin{equation*}
\begin{aligned}
\mean{2B_i}a(\tilde{x}_{i})|Dw_{i,j}-Dv_{i,j}|^{q}\,dx \overset{\eqref{Vprop}}&{\le} c\mean{2B_i}a(\tilde{x}_{i})|V_{q}(Dw_{i,j})-V_{q}(Dv_{i,j})|^{2}\,dx \\
\overset{\eqref{pq.phase.I}}&{\le} \frac{c}{K}(|Dw_{i,j}|^{p}+a(\tilde{x}_{i})|Dw_{i,j}|^{q})\,dx.
\end{aligned}
\end{equation*}
Combining the above three displays, and recalling $K \ge 4$, we arrive at
\begin{equation}\label{comp.pq.energy}
\mean{2B_i}H_{\tilde{x}_{i}}(|Dw_{i,j}-Dv_{i,j}|)\,dx \le \frac{c}{K^{p/2}}\mean{2B_i}H_{\tilde{x}_{i}}(|Dw_{i,j}|)\,dx
\end{equation}
for some $c\equiv c(n,p,q,\nu,L)$, where we have used the notations introduced in \eqref{def.Hx}. 
Observe that both $H_{\tilde{x}_{i}}$ and $H_{\tilde{x}_{i}} \circ h_{\tilde{x}_{i}}^{-1}$ are convex. 
We also note that, due to \eqref{a.equiv}, we can apply Lemma \ref{lem.pq.phase} to $w_{i,j}$, which gives
\[ \mean{2B_i}H_{\tilde{x}_{i}}(|Dw_{i,j}|)\,dx \le cH_{\tilde{x}_{i}}\left(\mean{4B_i}|Dw_{i,j}|\,dx\right)  \]
for some $c\equiv c(n,p,q,\nu,L)$. 
Applying this and Jensen's inequality to \eqref{comp.pq.energy}, we obtain
\begin{equation}\label{conclusion.pq}
\begin{aligned}
\mean{2B_{i}}h_{\tilde{x}_{i}}(|Dw_{i,j}-Dv_{i,j}|)\,dx & \le (h_{\tilde{x}_{i}}\circ H_{\tilde{x}_{i}}^{-1})\left(\mean{2B_{i}}H_{\tilde{x}_{i}}(|Dw_{i,j}-Dv_{i,j}|)\,dx\right) \\
& \le c(h_{\tilde{x}_{i}}\circ H_{\tilde{x}_{i}}^{-1})\left(\frac{1}{K^{p/2}}\mean{2B_{i}}H_{\tilde{x}_{i}}(|Dw_{i,j}|)\,dx\right) \\
& \le \frac{c}{K^{p(p-1)/(2q)}}\mean{4B_{i}}h_{\tilde{x}_{i}}(|Dw_{i,j}|)\,dx \\
\overset{\eqref{a.equiv}}&{\le} \frac{\tilde{c}}{K^{p(p-1)/(2q)}}\mean{4B_i}h(x,|Dw_{i,j}|)\,dx
\end{aligned}
\end{equation}
for some $\tilde{c}\equiv \tilde{c}(n,p,q,\nu,L)$.

\textit{Step 6: Estimates in the $p$-phase. } We next assume $\eqref{alt.comp2}_{2}$, which along with \eqref{a.holder} implies
\begin{equation}\label{p.phase.am}
a(\tilde{x}_{i}) \le \inf_{x \in 2B_i}a(x) + [a]_{\alpha}(4\varrho_i)^{\alpha} \le (4+K)[a]_{\alpha}\varrho_i^{\alpha}.
\end{equation} 
We start with applying the obvious inequality $\osc_{2B_i}a \le a(\tilde{x}_{i})$ to the right-hand side of \eqref{comp.I}, which gives
\begin{equation}\label{p.phase.I}
\begin{aligned}
& \mean{2B_i}\left(|V_{p}(Dw_{i,j})-V_{p}(Dv_{i,j})|^{2}+a(\tilde{x}_{i})|V_{q}(Dw_{i,j})-V_{q}(Dv_{i,j})|^{2}\right)\,dx \\
& \le c\mean{2B_i}a(\tilde{x}_{i})|Dw_{i,j}|^{q-1}|Dv_{i,j}-Dw_{i,j}|\,dx.
\end{aligned}
\end{equation}
We then proceed as
\begin{equation*}
\begin{aligned}
& \mean{2B_i}|Dw_{i,j}-Dv_{i,j}|^{p}\,dx \\
\overset{\eqref{Vprop}}&{\le} c_{\varepsilon}\mean{2B_i}|V_{p}(Dw_{i,j})-V_{p}(Dv_{i,j})|^{2}\,dx + \varepsilon^{p/(p-1)}\mean{2B_i}|Dw_{i,j}|^{p}\,dx \\
\overset{\eqref{p.phase.I}}&{\le} c_{\varepsilon}\mean{2B_i}a(\tilde{x}_{i})|Dw_{i,j}|^{q-1}|Dw_{i,j}-Dv_{i,j}|\,dx + \varepsilon^{p/(p-1)}\mean{2B_i}|Dw_{i,j}|^{p}\,dx \\
\overset{\text{Young}}&{\le} c_{\varepsilon}c_{\tilde{\varepsilon}}\mean{2B_i}\left(a(\tilde{x}_{i})|Dw_{i,j}|^{q-1}\right)^{p/(p-1)}\,dx + c_{\varepsilon}\tilde{\varepsilon}\mean{2B_i}|Dw_{i,j}-Dv_{i,j}|^{p}\,dx \\
& \qquad + \varepsilon^{p/(p-1)}\mean{2B_i}|Dw_{i,j}|^{p}\,dx
\end{aligned}
\end{equation*}
whenever $\varepsilon, \tilde{\varepsilon} \in (0,1)$, where $c_{\varepsilon} \equiv c_{\varepsilon}(n,p,q,\nu,L,\varepsilon)$ and $c_{\tilde{\varepsilon}} \equiv c_{\tilde{\varepsilon}}(n,p,q,\nu,L,\tilde{\varepsilon})$. Choosing $\tilde{\varepsilon} = 1/(2c_{\varepsilon})$, reabsorbing terms and then applying H\"older's inequality, we arrive at
\begin{equation}\label{p.phase.p-1.start}
\begin{aligned}
& \mean{2B_i}|Dw_{i,j}-Dv_{i,j}|^{p-1}\,dx \le \left(\mean{2B_i}|Dw_{i,j}-Dv_{i,j}|^{p}\,dx\right)^{(p-1)/p} \\
& \le c_{\varepsilon}a(\tilde{x}_{i})\left(\mean{2B_i}|Dw_{i,j}|^{\frac{(q-1)p}{p-1}}\,dx \right)^{(p-1)/p} + 2\varepsilon\left(\mean{2B_i}|Dw_{i,j}|^{p}\,dx\right)^{(p-1)/p} 
\end{aligned}
\end{equation}
for some $c_{\varepsilon} \equiv c_{\varepsilon}(n,p,q,\nu,L,\varepsilon)$. 
We now observe that
\begin{equation*}
\eqref{rate.sola}, \, \alpha \le 1 \;\; \Longrightarrow \;\; \frac{(q-1)p}{p-1} < \frac{(n-1+\alpha)p}{n-1} < \frac{np}{n-2\alpha},
\end{equation*}
with the understanding that $np/(n-2\alpha) = \infty$ if $n=2$ and $\alpha=1$. This fact together with \eqref{p.phase.am} enables us to apply Lemma \ref{lem:higher} with the choice $\tilde{q} = (q-1)p/(p-1)$, which gives
\begin{equation*}
\begin{aligned}
& a(\tilde{x}_{i})\left(\mean{2B_i}|Dw_{i,j}|^{\frac{(q-1)p}{p-1}}\,dx\right)^{(p-1)/p} \\
\overset{\eqref{p.phase.am},\eqref{higher.est}}&{\le} c\varrho_{i}^{\alpha}\left(\mean{3B_i}|Dw_{i,j}|^{\kappa}\,dx\right)^{(q-1)/\kappa} \\
& = c\varrho_{i}^{\alpha}\left(\mean{3B_i}|Dw_{i,j}|^{\kappa}\,dx\right)^{(q-p)/\kappa + (p-1)/\kappa} \\
& \le c\varrho_{i}^{\alpha-n(q-p)/\kappa}\|Dw_{i,j}\|_{L^{\kappa}(3B_i)}^{q-p}\left(\mean{3B_i}|Dw_{i,j}|^{\kappa}\,dx\right)^{(p-1)/\kappa} \\
\overset{\eqref{higher.est}}&{\le} c\varrho_{i}^{\alpha-n(q-p)/\kappa}\|Dw_{i,j}\|_{L^{\kappa}(4B_i)}^{q-p}\mean{4B_i}|Dw_{i,j}|^{p-1}\,dx
\end{aligned}
\end{equation*}
for some $c \equiv c(\data, \|Dw_{i,j}\|_{L^{\kappa}(4B_i)},K)$. 
We also have
\begin{equation*}
\left(\mean{2B_i}|Dw_{i,j}|^{p}\,dx\right)^{(p-1)/p} \le \bar{c}\mean{4B_i}|Dw_{i,j}|^{p-1}\,dx
\end{equation*}
for some $\bar{c}\equiv \bar{c}(\data, \|Dw_{i,j}\|_{L^{\kappa}(4B_i)},K)$. 
Connecting the above two displays to \eqref{p.phase.p-1.start}, 
we get
\begin{equation*}
\begin{aligned}
\mean{2B_i}|Dw_{i,j}-Dv_{i,j}|^{p-1}\,dx \le \left[ c_{\varepsilon}\varrho_{i}^{\alpha-n(q-p)/\kappa}\|Dw_{i,j}\|_{L^{\kappa}(4B_i)}^{q-p} + 2\bar{c}\varepsilon \right]\mean{4B_i}|Dw_{i,j}|^{p-1}\,dx
\end{aligned}
\end{equation*}
for any $\varepsilon \in (0,1)$, where $c_{\varepsilon} \equiv c_{\varepsilon}(\data, \|Dw_{i,j}\|_{L^{\kappa}(4B_i)},K,\varepsilon)$. 
In this inequality, we choose $\varepsilon = \tilde{c}/(2\bar{c}K^{p(p-1)/(2q)})$ for the constant $\tilde{c} \equiv \tilde{c}(n,p,q,\nu,L)$ given in \eqref{conclusion.pq}; this gives
\begin{equation}\label{p.phase.p-1}
\begin{aligned}
& \mean{2B_i}|Dw_{i,j}-Dv_{i,j}|^{p-1}\,dx \\
& \le \left[c\varrho_{i}^{\alpha-n(q-p)/\kappa}\|Dw_{i,j}\|_{L^{\kappa}(4B_i)}^{q-p} + \frac{\tilde{c}}{K^{p(p-1)/(2q)}}\right]\mean{4B_i}|Dw_{i,j}|^{p-1}\,dx,
\end{aligned}
\end{equation}
where $c\equiv c(\data, \|Dw_{i,j}\|_{L^{\kappa}(4B_i)}, K)$. 
In a similar way, we also have
\begin{equation*}
\begin{aligned}
& \mean{2B_i}a(\tilde{x}_{i})|Dw_{i,j}-Dv_{i,j}|^{q}\,dx \\
\overset{\eqref{Vprop}}&{\le} c\mean{2B_i}a(\tilde{x}_{i})|V_{q}(Dw_{i,j})-V_{q}(Dv_{i,j})|^{2}\,dx + c\mean{2B_i}a(\tilde{x}_{i})|Dw_{i,j}|^{q}\,dx \\
\overset{\eqref{p.phase.I}}&{\le} c\mean{2B_i}a(\tilde{x}_{i})|Dw_{i,j}|^{q-1}|Dw_{i,j}-Dv_{i,j}|\,dx + c\mean{2B_i}a(\tilde{x}_{i})|Dw_{i,j}|^{q}\,dx \\
\overset{\text{Young}}&{\le} \frac{1}{2}\mean{2B_i}a(\tilde{x}_{i})|Dw_{i,j}-Dv_{i,j}|^{q}\,dx + c\mean{2B_i}a(\tilde{x}_{i})|Dw_{i,j}|^{q}\,dx
\end{aligned}
\end{equation*}
and so
\begin{equation}\label{p.phase.start}
\mean{2B_i}a(\tilde{x}_{i})|Dw_{i,j}-Dv_{i,j}|^{q}\,dx \le c\mean{2B_i}a(\tilde{x}_{i})|Dw_{i,j}|^{q}\,dx,
\end{equation}
where $c\equiv c(n,p,q,\nu,L)$. 
Subsequently, we estimate
\begin{equation*}
\begin{aligned}
\mean{2B_{i}}a(\tilde{x}_{i})|Dw_{i,j}-Dv_{i,j}|^{q-1}\,dx &\le [a(\tilde{x}_{i})]^{1/q}\left(\mean{2B_i}a(\tilde{x}_{i})|Dw_{i,j}-Dv_{i,j}|^{q}\,dx\right)^{(q-1)/q} \\
\overset{\eqref{p.phase.start}}& {\le} ca(\tilde{x}_{i})\left(\mean{2B_{i}}|Dw_{i,j}|^{q}\,dx\right)^{(q-1)/q} \\
\overset{\eqref{p.phase.am},\eqref{higher.est}}& {\le} c\varrho_{i}^{\alpha}\left(\mean{3B_{i}}|Dw_{i,j}|^{\kappa}\,dx\right)^{(q-1)/\kappa} \\
& = c\varrho_{i}^{\alpha}\left(\mean{3B_{i}}|Dw_{i,j}|^{\kappa}\,dx\right)^{(q-p)/\kappa + (p-1)/\kappa} \\
& \le c\varrho_{i}^{\alpha-n(q-p)/\kappa}\|Dw_{i,j}\|_{L^{\kappa}(3B_i)}^{q-p}\left(\mean{3B_{i}}|Dw_{i,j}|^{\kappa}\,dx\right)^{(p-1)/\kappa} \\
\overset{\eqref{higher.est}}&{\le} c\varrho_{i}^{\alpha-n(q-p)/\kappa}\|Dw_{i,j}\|_{L^{\kappa}(4B_i)}^{q-p}\mean{4B_i}|Dw_{i,j}|^{p-1}\,dx,
\end{aligned}
\end{equation*}
where $c\equiv c(\data,\|Dw_{i,j}\|_{L^{\kappa}(4B_i)}, K)$. Combining this inequality with \eqref{p.phase.p-1}, and using the fact that \eqref{R0.ini},  \eqref{comp.young} (with $\varepsilon=1$), \eqref{approx.Du.2} and \eqref{approx.mu} imply
\begin{equation}\label{Dw.data}
\begin{aligned}
\|Dw_{i,j}\|_{L^{\kappa}(4B_i)} & \le \|Dw_{i,j}-Du_j\|_{L^{\kappa}(4B_i)} + \|Du_j\|_{L^{\kappa}(4B_i)} \\
& \le c\varrho_{i}^{n-\kappa(n-1)/(p-1)}[|\mu_j|(4B_i)]^{\kappa/(p-1)} + \|Du_j\|_{L^{\kappa}(4B_i)} \\
& \le c(\data, \|Du\|_{L^{\kappa}(\Omega)}),
\end{aligned}
\end{equation} 
we arrive at
\begin{equation}\label{conclusion.p}
\mean{2B_i}h(\tilde{x}_{i} ,|Dw_{i,j}-Dv_{i,j}|)\,dx \le \left[ \frac{\tilde{c}}{K^{p(p-1)/(2q)}} + c_{K}\varrho_{i}^{\sigma} \right]\mean{4B_{i}}h(x,|Dw_{i,j}|)\,dx
\end{equation}
for some constants $\tilde{c}\equiv \tilde{c}(n,p,q,\nu,L)$ and $c_{K}\equiv c_{K}(\data, \|Du\|_{L^{\kappa}(\Omega)}, K)$, where
\begin{equation}\label{def.sigma} 
\sigma \coloneqq \alpha-\frac{n(q-p)}{\kappa} > 0 \quad \text{by }\, \eqref{kappa.range}. 
\end{equation}

\textit{Step 7: Matching the two phases and comparison estimates. }
Combining \eqref{conclusion.pq} and \eqref{conclusion.p}, we have that
\begin{equation*}
\begin{aligned}
\mean{2B_i}h(\tilde{x}_{i},|Dw_{i,j}-Dv_{i,j}|)\,dx 
\le \left[\frac{\tilde{c}}{K^{p(p-1)/(2q)}} + c_{K}\varrho_{i}^{\sigma}\right]\mean{4B_i}h(x,|Dw_{i,j}|)\,dx
\end{aligned}
\end{equation*}
holds with $\tilde{c}\equiv \tilde{c}(n,p,q,\nu,L)$ and $c_{K} \equiv c_{K}(\data, \|Du\|_{L^{\kappa}(\Omega)}, K)$, in both cases $\eqref{alt.comp2}_1$ and $\eqref{alt.comp2}_2$. 
Here, if $R_0$ is so small that
\begin{equation}\label{R0.ini2}
c_{K}R_{0}^{\sigma} \le 1,
\end{equation}
then
\begin{equation*}
\begin{aligned}
& \mean{2B_i}h(\tilde{x}_{i},|Dw_{i,j}-Dv_{i,j}|)\,dx \\
& \le \max\left\{2^{q-2},1\right\}\left[\frac{\tilde{c}}{K^{p(p-1)/(2q)}} + c_{K}\varrho_{i}^{\sigma}\right]\left(\mean{4B_i}h(x,|Dw_{i,j}-Du_j|)\,dx + \mean{4B_i}h(x,|Du_j|)\,dx\right) \\
\overset{\eqref{comp1}}&{\le} c\left[\frac{|\mu|(\overline{8B_i})}{(8\varrho_i)^{n-1}}\right] + \max\left\{2^{q-2},1\right\}\left[\frac{\tilde{c}}{K^{p(p-1)/(2q)}} + c_{K}\varrho_{i}^{\sigma}\right]\mean{8B_i}h(x,|Du_{j}|)\,dx 
\end{aligned}
\end{equation*}
holds with $c\equiv c(n,p,q,\nu,L)$, as $K \ge 4$.
Combining the above display with \eqref{comp1}, we arrive at
\begin{equation}\label{comp.comb}
\begin{aligned}
& \mean{2B_i}h(x,|Du_{j}-Dv_{i,j}|)\,dx \\
& \le \max\left\{2^{q-2},1\right\}\left(\mean{2B_i}h(x,|Du_j - Dw_{i,j}|)\,dx + \mean{2B_i}h(\tilde{x}_{i},|Dw_{i,j}-Dv_{i,j}|)\,dx\right) \\
& \le  c_{*}\varepsilon_{0}^{\frac{p-2}{p-1}}\left[\frac{|\mu|(\overline{8B_i})}{(8\varrho_i)^{n-1}}\right] + \max\left\{2^{2(q-2)},1\right\}\left[\frac{\tilde{c}}{K^{p(p-1)/(2q)}} + c_{K}R_{0}^{\sigma} + \varepsilon_{0} \right]\mean{8B_i}h(x,|Du_j|)\,dx
\end{aligned}
\end{equation}
whenever $\varepsilon_0 \in (0,1]$, where $c_{*}\equiv c_{*}(\data,\|Du\|_{L^{\kappa}(\Omega)})$, $\tilde{c}\equiv \tilde{c} (n,p,q,\nu,L)$ and $ c_{K} \equiv c_{K} (\data , \|Du\|_{L^{\kappa}(\Omega)} , K)$. 
With these constants and $c_{0} \equiv c_{0}(n)$ given in \eqref{ineq.M1},  we now introduce the notation
\begin{equation}\label{def.S}
 S(\varepsilon_0, R_0, K,M) \coloneqq \max\left\{2^{4(q-2)},1\right\}\left[\frac{\tilde{c}}{K^{p(p-1)/(2q)}} + c_{K}R_{0}^{\sigma} + \varepsilon_{0} \right] + \frac{c_{0}c_{*}\varepsilon_{0}^{(p-2)/(p-1)}+1}{M}.
\end{equation}
Then, in light of \eqref{ineq.M1}, \eqref{lambda.range}, \eqref{exit2}, \eqref{approx.Du}, and \eqref{comp.comb}, we have that
\begin{equation}\label{final.comp}
\mean{2B_i}h(x,|Du_{j}-Dv_{i,j}|)\,dx \le S(\varepsilon_0, R_0,K,M)\lambda
\end{equation}
holds for any ball $B_i$ from the covering given in \eqref{def.Bi}. 

\textit{Step 8: The two phases at a different threshold. }
Here we show that
\begin{equation}\label{h.Dv.lambda}
\mean{2B_i}h(\tilde{x}_{i},|Dv_{i,j}|)\,dx  \le c\lambda
\end{equation}
holds for a constant $c\equiv c(\data,\|Du\|_{L^{\kappa}(\Omega)})$.  To this aim, we start by estimating
\begin{equation*}
\mean{2B_i}h(\tilde{x}_{i},|Dv_{i,j}|)\,dx \le c\mean{2B_i}h(\tilde{x}_{i},|Dw_{i,j}-Dv_{i,j}|)\,dx + c\mean{2B_i}h(\tilde{x}_{i},|Dw_{i,j}|)\,dx.
\end{equation*}
Then, we consider the following two alternatives:
\begin{equation}\label{alt.2Bi}
\inf_{x \in 2B_i}a(x) > 10[a]_{\alpha}\varrho_{i}^{\alpha} \qquad \text{and} \qquad \inf_{x\in 2B_i}a(x) \le 10[a]_{\alpha}\varrho_{i}^{\alpha},
\end{equation}
which are nothing but $\eqref{alt.comp2}_{1}$ and $\eqref{alt.comp2}_{2}$, respectively, with $K=10$. 

If $\eqref{alt.2Bi}_{1}$ holds, then we have
\begin{equation*}
\mean{2B_i}h(\tilde{x}_{i},|Dw_{i,j}-Dv_{i,j}|)\,dx \overset{\eqref{conclusion.pq}}{\le} c\mean{4B_i}h(x,|Dw_{i,j}|)\,dx
\end{equation*}
and
\begin{equation}\label{h.Dw.lambda}
\begin{aligned}
& \mean{2B_i}h(\tilde{x}_{i},|Dw_{i,j}|)\,dx \overset{\eqref{a.equiv}}{\le} c\mean{4B_i}h(x,|Dw_{i,j}|)\,dx \\
& \le c\mean{4B_i}h(x,|Du_j - Dw_{i,j}|)\,dx + c\mean{4B_i}h(x,|Du_j|)\,dx \\
\overset{\eqref{comp1}}& {\le} c\left[\frac{|\mu|(\overline{8B_i})}{(8\varrho_i)^{n-1}}\right] + c\mean{8B_i}h(x,|Du-Du_j|)\,dx + c\mean{8B_i}h(x,|Du|)\,dx \\
& \le c\lambda
\end{aligned}
\end{equation}
for some $c\equiv c(\data,\|Du\|_{L^{\kappa}(\Omega)})$, where we have also used \eqref{ineq.M1}, \eqref{def.lambda0}, \eqref{lambda.range}, and \eqref{approx.Du} for the last inequality. 
The above two displays lead to \eqref{h.Dv.lambda}.

If $\eqref{alt.2Bi}_{2}$ holds, then we estimate
\begin{equation*}
\begin{aligned}
a(\tilde{x}_{i})\mean{2B_i}|Dw_{i,j}|^{q-1}\,dx \overset{\eqref{p.phase.am},\eqref{higher.est}}&{\le} c\varrho_{i}^{\alpha}\left(\mean{3B_i}|Dw_{i,j}|^{\kappa}\,dx\right)^{(q-1)/\kappa} \\
& = c\varrho_{i}^{\alpha}\left(\mean{3B_i}|Dw_{i,j}|^{\kappa}\,dx\right)^{(q-p)/\kappa + (p-1)/\kappa} \\
& \le c\varrho_{i}^{\alpha-n(q-p)/\kappa}\|Dw_{i,j}\|^{q-p}_{L^{\kappa}(3B_i)}\left(\mean{3B_i}|Dw_{i,j}|^{\kappa}\,dx\right)^{(p-1)/\kappa} \\
\overset{\eqref{kappa.range},\eqref{Dw.data},\eqref{higher.est}}&{\le} c\mean{4B_i}|Dw_{i,j}|^{p-1}\,dx
\end{aligned}
\end{equation*}
for some $c\equiv c(\data,\|Du\|_{L^{\kappa}(\Omega)})$, which gives \eqref{h.Dw.lambda} in this case as well.
Using this inequality and \eqref{conclusion.p}, we again conclude with \eqref{h.Dv.lambda}.

\textit{Step 9: A priori estimates for $Dv_{i,j}$. } 
Note that estimate \eqref{h.Dv.lambda} is independent of the cases $\eqref{alt.2Bi}_1$ and $\eqref{alt.2Bi}_2$ coming into the play.  We then apply Lemma \ref{lem.Lip} to $v_{i,j}$, with the choice $a_{0} \equiv a(\tilde{x}_{i})$, which gives
\begin{equation*}
 \sup_{B_i}h(\tilde{x}_{i},|Dv_{i,j}|) \le c\mean{2B_i}h(\tilde{x}_{i},|Dv_{i,j}|)\,dx \le c_{l}\lambda
 \end{equation*}
for some $c_{l}\equiv c_{l}(\data,\|Du\|_{L^{\kappa}(\Omega)})$. 
This together with the definition of $\tilde{x}_{i}$ implies
\begin{equation}\label{lip.lambda}
\sup_{B_i}h(\cdot,|Dv_{i,j}|) \le c_{l}\lambda.
\end{equation}

\textit{Step 10: Estimates on upper level sets. } 
Here and in the following, we use the notation
\[ c_{q} \coloneqq \max\left\{3^{q-2},1\right\}. \] 
By using \eqref{lip.lambda}, we have
\begin{equation*}
\begin{aligned}
& c_{q}c_{l} \lambda 
|E( B_{i} , 2c_{q}c_{l}\lambda )|
+ \frac{1}{2} \int_{E( B_{i} , 2c_{q}c_{l} \lambda )}
h(x,|Du|)\,dx \\
& \le \int_{E( B_{i} , 2c_{q}c_{l} \lambda )} h(x,|Du|)\,dx \\
& \le c_{q}\int_{B_i}h(x,|Du-Du_{j}|)\,dx + c_{q}\int_{B_i}h(x,|Du_{j}-Dv_{i,j}|)\,dx  + c_{q}\int_{E( B_{i} , 2c_{q}c_{l} \lambda )} h(x,|Dv_{i,j}|)\,dx \\
& \le c_{q}\int_{B_i}h(x,|Du-Du_j|)\,dx + c_{q}\int_{B_i}h(x,|Du_j - Dv_{i,j}|)\,dx  + c_{q}c_{l} \lambda| E( B_{i} , 2c_{q}c_{l} \lambda ) |
\end{aligned}
\end{equation*}
and therefore
\begin{equation*}
\begin{aligned}
& \int_{E( B_{i} , 2c_{q}c_{l} \lambda )} h(x,|Du|)\,dx \\
& \le 2c_{q}\int_{B_i}h(x,|Du-Du_j|)\,dx + 2c_{q}\int_{B_i}h(x,|Du_j - Dv_{i,j}|)\,dx \\
& \le 2c_{q}|8B_i|\left(\mean{8B_i}h(x,|Du-Du_{j}|)\,dx + \mean{2B_i}h(x,|Du_{j}-Dv_{i,j}|)\,dx\right).
\end{aligned}
\end{equation*}
Applying \eqref{approx.Du} and \eqref{final.comp} to each term in the right-hand side, we arrive at
\begin{equation}\label{hDu.Bi}
\int_{E( B_{i} , 2c_{q}c_{l} \lambda )} h(x,|Du|)\,dx \le 80^{n}c_{q}S(\varepsilon_0,R_0,K,M)\lambda|\tilde{B}_i|.
\end{equation}

We now establish a uniform estimate for $|\tilde{B}_i|$. Observe that $\eqref{exit2}_{1}$ implies
\begin{equation*}
|\tilde{B}_i| = \frac{1}{\lambda}\int_{\tilde{B}_i}[h(x,|Du|) + M\mathbf{M}_{1}(\mu)]\,dx.
\end{equation*}
In particular, either
\begin{equation}\label{hm}
\frac{1}{\lambda}\int_{\tilde{B}_i}h(x,|Du|)\,dx \ge \frac{1}{2}|\tilde{B}_i| \qquad \text{or} \qquad \frac{1}{\lambda}\int_{\tilde{B}_i}M\mathbf{M}_{1}(\mu)\,dx \ge \frac{1}{2}|\tilde{B}_i|
\end{equation}
must hold. If $\eqref{hm}_1$ holds, then we have
\begin{equation*} 
|\tilde{B}_i| \le \frac{2}{\lambda}\int_{\tilde{B}_i}h(x,|Du|)\,dx 
\le \frac{2}{\lambda}\int_{E( \tilde{B}_{i} , \lambda/4 )} h(x,|Du|) \, dx + \frac{1}{2} |\tilde{B}_i|
\end{equation*}
and so
\[ |\tilde{B}_i| \le \frac{4}{\lambda}\int_{E(\tilde{B}_i , \lambda/4)}h(x,|Du|)\,dx. \]
Similarly, if $\eqref{hm}_2$ holds, then we have
\[ |\tilde{B}_i| \le \frac{4}{\lambda}\int_{\mathcal{E}(\tilde{B}_i , \lambda/(4M))} M\mathbf{M}_1(\mu)\,dx. \]
The above two displays imply
\begin{equation*}
|\tilde{B}_i| \le \frac{4}{\lambda}\int_{E(\tilde{B}_i , \lambda/4)} h(x,|Du|)\,dx 
+ \frac{4}{\lambda}\int_{\mathcal{E}(\tilde{B}_i , \lambda/(4M))} M\mathbf{M}_{1}(\mu)\,dx.
\end{equation*}
Connecting this inequality to \eqref{hDu.Bi}, we arrive at
\begin{equation}\label{upper.1st}
\begin{aligned}
\int_{E(B_i , 2 c_{q}c_{l}\lambda )}h(x,|Du|)\,dx 
& \le 4\cdot80^{n}c_{q}S(\varepsilon_0,R_0,K,M)\int_{ E( \tilde{B}_i , \lambda/4 )}h(x,|Du|)\,dx \\
& \quad + 4\cdot80^{n}c_{q}S(\varepsilon_0,R_{0},K,M)\int_{\mathcal{E}( \tilde{B}_i , \lambda/(4M) )}M\mathbf{M}_{1}(\mu)\,dx.
\end{aligned}
\end{equation}
We then recall \eqref{covering} and the fact that the balls $\{\tilde{B}_i\}$ are disjoint. Accordingly, after changing variables, we sum up \eqref{upper.1st} over the covering $\{B_i\}$ in order to have
\begin{equation}\label{upper.2nd}
\begin{aligned}
\int_{E(B_{r_1},\lambda)}h(x,|Du|)\,dx & \le 4\cdot80^{n}c_{q}S(\varepsilon_0,R_0,K,M)\int_{E(B_{r_2},\lambda/(8c_{q}c_{l}))}h(x,|Du|)\,dx \\
& \quad + 4\cdot80^{n}c_{q}S(\varepsilon_0,R_0,K,M)\int_{
B_{r_2}}M\mathbf{M}_{1}(\mu)\,dx
\end{aligned}
\end{equation}
whenever
\[ \lambda \ge 2c_{q}c_{l}\lambda_{0} = 2c_{q}c_{l}\frac{20^{n}r_{2}^{n}}{(r_{2}-r_{1})^n}\mean{B_{r_2}}[h(x,|Du|) + M\mathbf{M}_{1}(\mu)]\,dx. \]

\textit{Step 11: Integration and conclusion. }
We now conclude the proof by integrating on level sets.
Let us fix $\gamma \in (1, \infty)$.
We consider, for $t \ge 0$, the truncated function
\[ [h(x,|Du|)]_t \coloneqq \min\{h(x,|Du|), t\}. \]
For $t \ge 4c_{q}c_{l}\lambda_0$, we multiply \eqref{upper.2nd} by $\lambda^{\gamma-2}$ and then integrate the resulting estimate with respect to $\lambda$ in order to get
\begin{equation}\label{int.start}
\begin{aligned}
& \int_{2c_{q}c_{l}\lambda_0}^{t}\lambda^{\gamma-2}\int_{E(B_{r_{1}},\lambda)}h(x,|Du|)\,dx\,d\lambda \\
& \le 4\cdot80^{n}c_{q}S(\varepsilon_0,R_0,K,M)\int_{2c_{q}c_{l}\lambda_0}^{t}\lambda^{\gamma-2}\int_{E(B_{r_2},\lambda/(8c_{q}c_{l}))}h(x,|Du|)\,dx\,d\lambda \\
& \quad + 4\cdot80^{n}c_{q}S(\varepsilon_0,R_0,K,M)\int_{2c_{q}c_{l}\lambda_0}^{t}\lambda^{\gamma-2}\int_{\mathcal{E}(B_{r_2},\lambda/(8c_{q}c_{l}M))}M\mathbf{M}_1(\mu)\,dx\,d\lambda.
\end{aligned}
\end{equation}
Observe that Fubini's theorem implies
\begin{equation}\label{int.2nd}
\begin{aligned}
& \int_{2c_{q}c_{l}\lambda_0}^{t}\lambda^{\gamma-2}\int_{E(B_{r_{1}},\lambda)}h(x,|Du|)\,dx\,d\lambda \\
& = \frac{1}{\gamma-1}\int_{B_{r_1}}[h(x,|Du|)]^{\gamma-1}_{t}h(x,|Du|)\,dx - \int_{0}^{2c_{q}c_{l}\lambda_0}\lambda^{\gamma-2}\int_{E(B_{r_1},\lambda)}h(x,|Du|)\,dx\,d\lambda
\end{aligned}
\end{equation}
and, in a similar way, 
\begin{equation}\label{int.22}
\begin{aligned}
& \int_{2c_{q}c_{l}\lambda_0}^{t}\lambda^{\gamma-2}\int_{E(B_{r_2},\lambda/(8c_{q}c_l))}h(x,|Du|)\,dx\,d\lambda \\
& \le \frac{(8c_{q}c_l)^{\gamma-1}}{\gamma-1}\int_{B_{r_2}}[h(x,|Du|)]^{\gamma-1}_{t/(8c_{q}c_l)}h(x,|Du|)\,dx \\
& \le \frac{(8c_{q}c_l)^{\gamma-1}}{\gamma-1}\int_{B_{r_2}}[h(x,|Du|)]^{\gamma-1}_{t}h(x,|Du|)\,dx.
\end{aligned}
\end{equation}
Note that, in the above estimate, we have changed variables and used the inequality 
\[ [h(x,|Du|)]^{\gamma-1}_{t/(8c_{q}c_l)} \le [h(x,|Du|)]^{\gamma-1}_{t}, \]
which is valid since $8c_{q}c_l >1$. 
On the other hand, recalling the definition of $\lambda_0$ given in \eqref{def.lambda0}, the second term in the right-hand side of \eqref{int.2nd} is estimated as
\begin{equation}\label{int.lhs}
\begin{aligned}
\int_{0}^{2c_{q}c_{l}\lambda_0}\lambda^{\gamma-2}\int_{E(B_{r_1},\lambda)}h(x,|Du|)\,dx\,d\lambda & \le \int_{0}^{2c_{q}c_{l}\lambda_0}\lambda^{\gamma-2}\,d\lambda \int_{B_{r_2}}h(x,|Du|)\,dx \\
& \le \frac{(2c_{q}c_l)^{\gamma-1}\lambda_{0}^{\gamma-1}}{\gamma-1}\int_{B_{r_2}}h(x,|Du|)\,dx \\
& \le \frac{(2c_{q})^{\gamma-1}c_{l}^{\gamma-1}}{\gamma-1}\lambda_{0}^{\gamma}|B_{r_2}|.
\end{aligned}
\end{equation}
As for the integral of $\mathbf{M}_1(\mu)$, we again use Fubini's theorem to have
\begin{equation}\label{int.23}
\begin{aligned}
& \int_{2c_{q}c_{l}\lambda_0}^{t}\lambda^{\gamma-2}\int_{\mathcal{E}(B_{r_2},\lambda/(8c_{q}c_l M))}\mathbf{M}_1(\mu)\,dx\,d\lambda \\
& \le \int_{0}^{\infty}\lambda^{\gamma-2}\int_{\mathcal{E}(B_{r_2},\lambda/(8c_{q}c_l M))}\mathbf{M}_1(\mu)\,dx\,d\lambda \le \frac{(8c_{q}c_l M)^{\gamma-1}}{\gamma-1}\int_{B_{r_2}}[\mathbf{M}_1(\mu)]^{\gamma}\,dx.
\end{aligned}
\end{equation}
Plugging \eqref{int.2nd}, \eqref{int.22}, \eqref{int.lhs} and \eqref{int.23} into \eqref{int.start}, and then performing elementary manipulations, we arrive at
\begin{equation}\label{int.3rd}
\begin{aligned}
& \mean{B_{r_1}}[h(x,|Du|)]^{\gamma-1}_{t}h(x,|Du|)\,dx \\
& \le c_{f}^{\gamma}c_{l}^{\gamma-1}S(\varepsilon_0,R_0,K,M)\mean{B_{r_2}}[h(x,|Du|)]^{\gamma-1}_{t}h(x,|Du|)\,dx \\
& \quad + c_{f}^{\gamma}c_{l}^{\gamma-1}M^{\gamma-1}S(\varepsilon_0,R_0,K,M)\mean{B_{r_2}}[\mathbf{M}_1(\mu)]^{\gamma}\,dx + c_{f}^{\gamma}c_{l}^{\gamma-1}\lambda_{0}^{\gamma},
\end{aligned}
\end{equation}
where $c_{l}\equiv c_{l}(\data,\|Du\|_{L^{\kappa}(\Omega)}) \ge 1$ is the constant given in \eqref{lip.lambda} and $c_{f} \equiv c_{f}(n,q) \ge 1$; note that we have also used the inequality $|B_{r_2}|/|B_{r_1}| \le 2^n$.

We now deal with the quantity $S(\varepsilon_0,R_0,K,M)$ defined in \eqref{def.S}; note that all the above estimates are valid for any choices of $R_0 \in (0,1)$ satisfying $R_0 \le R_*$ (recall \eqref{R0.ini}) and \eqref{R0.ini2}, $K \ge 4$ and $M \ge 1$, with all the constants, except for the ones appearing in the definition  of $S(\varepsilon_0,R_0,K,M)$, remaining uniformly bounded and ultimately depending only on $\data$ and $\|Du\|_{L^{\kappa}(\Omega)}$.
Here we fix $\varepsilon_0$, $R_0$, $K$ and $M$ in order to have 
\begin{equation}\label{fix.para}
c_{f}^{\gamma}c_{l}^{\gamma-1}S(\varepsilon_0,R_0,K,M) \le \frac{1}{2}.
\end{equation}
More precisely, we first fix $K \equiv K(\data, \|Du\|_{L^{\kappa}(\Omega)} ,\gamma) \ge 4$ and $\varepsilon \equiv \varepsilon(\data, \|Du\|_{L^{\kappa}(\Omega)} ,\gamma) \in (0,1)$ as
\begin{equation}\label{fix.K}
K \coloneqq \left(2^{4q}c_{f}^{\gamma}c_{l}^{\gamma}\tilde{c}\right)^{2q/[p(p-1)]} \qquad \text{and} \qquad \varepsilon = \frac{1}{2^{4q}c_{f}^{\gamma}c_{l}^{\gamma}},
\end{equation}
where $\tilde{c}\equiv \tilde{c}(n,p,q,\nu,L)$ appears in \eqref{comp.comb}. This in turn determines the constants $c_K$ and $c_{\varepsilon}$ in \eqref{comp.comb} as functions of $\data, \|Du\|_{L^{\kappa}(\Omega)}$ and $\gamma$ only. We then choose $M \equiv M(\data, \|Du\|_{L^{\kappa}(\Omega)} ,\gamma) \ge 1$ as
\begin{equation}\label{fix.M}
M \coloneqq 16c_{f}^{\gamma}c_{l}^{\gamma}c_{0}c_{*}\varepsilon_{0}^{(p-2)/(p-1)},
\end{equation}
where $c_{*} \equiv c_{*}(\data,\|Du\|_{L^{\kappa}(\Omega)})$ and and $c_{0} \equiv c_{0}(n)$ are given in \eqref{comp.comb} and \eqref{ineq.M1}, resepectively.
We finally choose $R_0 \equiv R_0 (\data,|\mu|(\Omega), \|Du\|_{L^{\kappa}(\Omega)},\gamma) \in (0,1)$ small enough to have $R_0 \le R_*$  and 
\begin{equation}\label{fix.R0}
R_0 \le \left(\frac{1}{2^{4q}c_{f}^{\gamma}c_{l}^{\gamma}c_{K}}\right)^{1/\sigma},
\end{equation}
where $\sigma \equiv \sigma (n,p,q,\alpha)$ is given in \eqref{def.sigma}; in particular, we also have \eqref{R0.ini2}. 
The choices of parameters in \eqref{fix.K}, \eqref{fix.M}, and \eqref{fix.R0} yield \eqref{fix.para}. Combining this with \eqref{int.3rd} and the definition of $\lambda_0$ in \eqref{def.lambda0}, we obtain
\[
\begin{aligned}
& \mean{B_{r_1}}[h(x,|Du|)]^{\gamma-1}_{t}h(x,|Du|)\,dx \\
& \le \frac{1}{2}\mean{B_{r_2}}[h(x,|Du|)]^{\gamma-1}_{t}h(x,|Du|)\,dx + c_{\gamma}\mean{B_{2R}}[\mathbf{M}_1(\mu)]^{\gamma}\,dx \\
& \quad + \frac{c^{\gamma}R^{n\gamma}}{(r_2 - r_1)^{n\gamma}}\left(\mean{B_{2R}}[h(x,|Du|)+M\mathbf{M}_1(\mu)]\,dx \right)^{\gamma},
\end{aligned}
\]
where $c\equiv c(\data, \|Du\|_{L^{\kappa}(\Omega)})$ and $c_{\gamma} \equiv c_{\gamma}(\data, \|Du\|_{L^{\kappa}(\Omega)} ,\gamma)$.
 Note in particular that the constants in the above estimate are independent of the parameter $t \ge 4c_{q}c_{l}\lambda_0$. 

We are now able to apply Lemma \ref{tech.lemma} with the choices $\ell = n\gamma$ and
\[ Z(s) \coloneqq \mean{B_s}[h(x,|Du|)]^{\gamma-1}_{t}h(x,|Du|)\,dx, \qquad R \le s \le 2R, \]
which is obviously a bounded function. Then, making a few elementary manipulations together with H\"older's inequality and recalling that $M \ge 1$ has been determined in \eqref{fix.M} as a function of $\data, \|Du\|_{L^{\kappa}(\Omega)}$ and $\gamma$ only, we arrive at
\[ \mean{B_R}[h(x,|Du|)]^{\gamma-1}_{t}h(x,|Du|)\,dx \le c^{\gamma}\left(\mean{B_{2R}}h(x,|Du|)\,dx\right)^{\gamma} + c_{\gamma}\mean{B_{2R}}[\mathbf{M}_1(\mu)]^{\gamma}\,dx. \]
Letting $t\to\infty$ in the above display, we conclude with
\[ \mean{B_R}[h(x,|Du|)]^{\gamma}\,dx \le c^{\gamma}\left(\mean{B_{2R}}h(x,|Du|)\,dx\right)^{\gamma} + c_{\gamma}\mean{B_{2R}}[\mathbf{M}_1(\mu)]^{\gamma}\,dx \]
whenever $R \le R_0$ with $R_0 \equiv R_0(\data,|\mu|(\Omega), \|Du\|_{L^{\kappa}(\Omega)}, \|Du\|_{L^{\gamma_0}(\Omega)}, \gamma) \in (0,1)$ satisfying $R_0 \le R_*$  and \eqref{fix.R0}, and where $c\equiv c(\data, \|Du\|_{L^{\kappa}(\Omega)})$ and $c_{\gamma}\equiv c_{\gamma}(\data, \|Du\|_{L^{\kappa}(\Omega)}, \gamma)$. This shows \eqref{main.est}, which along with a standard covering argument leads to \eqref{cz.implication}. The proof is finally complete. \qed

\section*{Declarations}

\subsection*{Acknowledgments} 

K. Song was supported by a KIAS individual grant (MG091702) at Korea Institute for Advanced Study.
Y. Youn was supported by National Research Foundation of Korea (NRF) grant funded by the Korean government [Grant No. RS-2025-00555316]. 
A. Zatorska-Goldstein was supported by the funding from the University of Warsaw in the framework of the Thematic Research Program \textit{Geometric Analysis: Methods and Applications GAMA25}

\subsection*{Conflict of interest} On behalf of all authors, the corresponding author states that there is no conflict of interest.

\subsection*{Data availability} Data sharing is not applicable to this article as no datasets were generated or analysed during the current study.

\end{document}